\documentclass[12pt]{amsart}
\usepackage{amssymb}

\newtheorem{theorem}{Theorem}[section]
\newtheorem{definition}[theorem]{Definition}
\newtheorem{proposition}[theorem]{Proposition}

\begin{document}

\title[Wishart matrices]{Block-modified Wishart matrices: the easy case}

\author{Teodor Banica}
\address{T.B.: Department of Mathematics, University of Cergy-Pontoise, F-95000 Cergy-Pontoise, France. {\tt teo.banica@gmail.com}}


\begin{abstract}
Associated to any complex Wishart matrix $W$ of parameters $(dn,dm)$ and any linear map $\varphi:M_n(\mathbb C)\to M_n(\mathbb C)$ is the ``block-modified'' matrix $\tilde{W}=(id\otimes\varphi)W$. Following some previous work with Nechita, we study here the asymptotic $*$-distribution of $\tilde{W}$, in the $d\to\infty$ limit, in the case where the modification map $\varphi$ is ``easy'', or more generally super-easy, in the quantum algebra/representation theory sense. Under suitable assumptions on $\varphi$ we obtain in this way a compound free Poisson law.
\end{abstract}

\maketitle

\section*{Introduction}

One of the guiding principles of Random Matrix Theory is that ``you can get anything inside a usual random matrix''. Indeed, the random matrices are known to encode a bewildering quantity of interesting mathematical and physical phenomena. See \cite{agz}, \cite{meh}.

A particularly efficient construction, which leads to a vast and beautiful combinatorial landscape, which has been barely explored so far, consists in looking at the asymptotic distribution of the block-modified Wishart matrices. To be more precise, consider a Wishart matrix $W$ of parameters $(dn,dm)$, with $d,n,m\in\mathbb N$. Associated to any linear map $\varphi:M_n(\mathbb C)\to M_n(\mathbb C)$ is then the ``block-modified'' matrix $\tilde{W}=(id\otimes\varphi)W$, and the problem is that of computing the $*$-distribution of $\tilde{W}$, in the $d\to\infty$ limit.

All this goes back to a 2012 paper of Aubrun \cite{aub}, where the case of the transposition map $\varphi(A)=A^t$ was considered. Aubrun's computation, leading to shifted semicircles, was generalized soon after in \cite{bn1}, with a result leading to certain free differences of free Poisson laws. A further generalization, leading to large classes of compound free Poisson laws, covering as well the result of Marchenko-Pastur \cite{mpa}, and some computations from \cite{cn1}, \cite{cn2} was worked out in \cite{bn2}. A number of supplementary results on the subject, sometimes obtained in different asymptotic regimes, are now available from \cite{anv}, \cite{fsn}, \cite{mpo}.

\bigskip

According to the general theory in \cite{bn2}, the 4 ``basic'' computations, corresponding to results in \cite{bn1}, \cite{cn1}, \cite{cn2}, \cite{mpa}, come from the following 4 diagrams: 
$$\pi_1=\begin{bmatrix}\circ&\bullet\\ \circ&\bullet\end{bmatrix}\quad,\quad
\pi_2=\begin{bmatrix}\circ&\bullet\\ \bullet&\circ\end{bmatrix}\quad,\quad
\pi_3=\begin{bmatrix}\circ&\circ\\ \bullet&\bullet\end{bmatrix}\quad,\quad
\pi_4=\begin{bmatrix}\circ&\circ\\ \circ&\circ\end{bmatrix}$$

To be more precise, these diagrams are the partitions $\pi\in P_{even}(2,2)$, between an upper row of 2 points and a lower row of 2 points, having all the blocks of even size. Now given such a partition $\pi\in P_{even}(2,2)$, we can associate to it the following linear map:
$$\varphi_\pi(e_{ac})=\sum_{bd}\delta_\pi\begin{pmatrix}a&c\\ b&d\end{pmatrix}e_{bd}$$

Here $\{e_{ac}\}$ is the standard basis of $M_n(\mathbb C)$, and $\delta_\pi\in\{0,1\}$ is given by $\delta_\pi=1$ when each block of $\pi$ contains equal indices, and $\delta_\pi=0$ otherwise.

With this convention, the linear maps associated to $\pi_1,\pi_2,\pi_3,\pi_4$ are as follows, where $A^\delta\in M_n(\mathbb C)$ denotes the diagonal of a given matrix $A\in M_n(\mathbb C)$:
$$\varphi_1(A)=A\quad,\quad
\varphi_2(A)=A^t\quad,\quad
\varphi_3(A)=Tr(A)1\quad,\quad
\varphi_4(A)=A^\delta$$

We recognize here the null block-modification, leading to the result of Marchenko-Pastur \cite{mpa}, the transposition, which leads to the results in \cite{aub}, \cite{bn1}, and then the trace and the diagonal restriction map, leading to the extra computations from \cite{cn1}, \cite{cn2}.

\bigskip

Summarizing, we have here an interesting phenomenon, relating partitions and Wishart matrices, worthly of exploration. Some preliminary results and explanations were obtained in \cite{bn2}, and some further work on the subject was performed afterwards, in \cite{anv}. 

Generally speaking, the idea in \cite{anv}, \cite{bn2} was to develop very general theories, based on the above observation. In \cite{bn2} such a theory was developed, by using a somewhat heavy diagrammatic formalism (the ``strings and beads'' operad). As for the work in \cite{anv}, the theory developed there is quite abstract too, using operator-valued free probability.

Our purpose here is to go somehow in an opposite, much more concrete direction. To be more precise, our philosophy will be that of using partitions $\pi\in P_{even}(2s,2s)$, with $s\in\mathbb N$ being arbitrary, and trying to get ``the best we can'' from these partitions.

\bigskip

As a starting point for the present considerations, the partitions $\pi\in P_{even}(2s,2s)$ are familiar objects, appearing in the representation theory of the hyperoctahedral group $H_N$. The noncrossing ones, $\pi\in NC_{even}(2s,2s)$, are familiar objects too, appearing in the representation theory of the hyperoctahedral quantum group $H_N^+$. All this goes back to our 2007 paper \cite{bbc} with Bichon and Collins, and the findings there, and from the subsequent paper \cite{bb+}, with Belinschi, Capitaine and Collins, led to the axiomatization of the ``easy quantum groups'', worked out in our 2009 paper with Speicher \cite{bsp}.

\bigskip

Roughly speaking, a compact Lie group or quantum group $G$ is called easy when its Tannakian dual comes from ``easy maps'', with these latter maps being associated to partitions. In the case of a partition $\pi\in P_{even}(2s,2s)$, corresponding to $G=H_N$, the associated easy map is, up to a certain contraction of the tensors, as follows:
$$\varphi_\pi(e_{a_1\ldots a_s,c_1\ldots c_s})=\sum_{b_1\ldots b_s}\sum_{d_1\ldots d_s}\delta_\pi\begin{pmatrix}a_1&\ldots&a_s&c_1&\ldots&c_s\\ b_1&\ldots&b_s&d_1&\ldots&d_s\end{pmatrix}e_{b_1\ldots b_s,d_1\ldots d_s}$$

Observe the similarity with the previous formula of $\varphi_\pi$, for $\pi\in P_{even}(2,2)$. In fact, what we have here is a generalization of the $s=1$ formula, valid at any $s\in\mathbb N$.

Summing up, we have a quite concrete problem to be solved, namely that of computing the asympotic distribution of the block-modified Wishart matrix $\tilde{W}=(id\otimes\varphi)W$, in the case where $\varphi=\varphi_\pi$ is an easy map, coming from a partition $\pi\in P_{even}(2s,2s)$.

\bigskip

We will perform our study at several generality levels, as follows:

\bigskip

{\bf I.} Following the previous work in \cite{bn2}, we will first study the partitions $\pi\in P_{even}(2s,2s)$ which are symmetric, with respect to the blockwise middle symmetry:
$$\begin{bmatrix}c_1&\ldots&c_s&a_1&\ldots&a_s\\ d_1&\ldots&d_s&b_1&\ldots&b_s\end{bmatrix}
\longleftrightarrow\begin{bmatrix}a_1&\ldots&a_s&c_1&\ldots&c_s\\ b_1&\ldots&b_s&d_1&\ldots&d_s\end{bmatrix}$$

Our study here will take advantage of the fact that we use now a lighter formalism. We will improve the previous findings in \cite{bn2}, with a stronger result on the subject.

\bigskip

{\bf II.} We will discuss then the $q=-1$ twisting of this result, by using the general theory developed in \cite{ba1}. The twisting procedure requires crossings, in order to be non-trivial, and at $s=1$ we only have one example, coming from the following partition:
$$\pi=\begin{bmatrix}\circ&\bullet\\ \bullet&\circ\end{bmatrix}$$

The corresponding twisted map is $\bar{\varphi}_\pi(A)=2A^\delta-A^t$, and at the block-modified Wishart matrix level, the law that we obtain is in fact the same as for $\varphi_\pi(A)=A^t$. 

Our main result here will be a generalization of this fact, to the case $\pi\in P_{even}(2s,2s)$. This is interesting, theoretically speaking, in view of \cite{ba1} and subsequent work.

\bigskip

{\bf III.} The twisted maps $\bar{\varphi}_\pi$ are obtained from the untwisted ones $\varphi_\pi$ by replacing the Kronecker symbols $\delta_\pi\in\{0,1\}$ by signed Kronecker symbols $\bar{\delta}_\pi\in\{-1,0,1\}$, the precise formula being as follows, with $\varepsilon:P_{even}\to\{-1,1\}$ being the signature map:
$$\bar{\varphi}_\pi(e_{a_1\ldots a_s,c_1\ldots c_s})=\sum_{\sigma\leq\pi}\varepsilon(\sigma)\sum_{\ker(^{ac}_{bd})=\sigma}e_{b_1\ldots b_s,d_1\ldots d_s}$$

One interesting question is that of further extending our formalism, by allowing the Kronecker symbols to take complex values, $\delta_\pi\in\mathbb T\cup\{0\}$. We believe that such an extended formalism can cover the free Bessel laws, constructed in \cite{bb+}, in $*$-moments. We will comment here on this question, with some preliminary observations.

\bigskip

There are of course many questions arising from the present work, some of them being of pure random matrix nature, and some other being in connection with the ongoing classification program for the easy quantum groups, and notably with the recent paper \cite{ba2}. We will comment on these problems at the end of the paper. 

The paper is organized as follows: in 1-2 we review the basic block modification theory from \cite{bn2}, with a few technical changes, in 3-4 we discuss the easy case, with a number of preliminary results, in 5-6 we state and prove our main results, regarding the easy case, in 7-8 we discuss the twisted easy case, and in 9-10 we discuss the free Bessel laws.

\bigskip

\noindent {\bf Acknowledgements.} I would like to thank O. Arizmendi and I. Nechita for useful discussions.

\section{Wishart matrices}

Consider a complex Wishart matrix of parameters $(dn,dm)$. In other words, we start with a $dn\times dm$ matrix $G$ having independent complex $\mathcal N(0,1)$ entries, and we set:
$$W=\frac{1}{dm}\,GG^*$$

This matrix has size $dn\times dn$, and is best thought of as being a $d\times d$ array of $n\times n$ matrices. We are interested here in the study of the ``block-modified'' versions of $W$, obtained by applying a given linear map $\varphi:M_n(\mathbb C)\to M_n(\mathbb C)$ to the $n\times n$ blocks.

With standard tensor product notations, this construction is as follows:

\begin{definition}
Associated to any Wishart matrix $W$ of parameters $(dn,dm)$ and any linear map $\varphi:M_n(\mathbb C)\to M_n(\mathbb C)$ is the ``block-modified'' matrix $\tilde{W}=(id\otimes\varphi)W$.
\end{definition}

We would like to compute the limiting $d\to\infty$ eigenvalue distribution of $\tilde{W}$, as a function of the modification map $\varphi$. Since $\tilde{W}$ is in general not self-adjoint, in order to have a complete picture, we must in fact compute its $*$-distribution. For this purpose, we use the moment method, or rather the $*$-moment method.

We use the following $*$-moment formalism:

\begin{definition}
The $*$-moments of a random matrix $X$, which depend on an integer $p\in\mathbb N$ and on a sequence of exponents $e_1,\ldots,e_p\in\{1,*\}$, are given by:
$$M^e(X)=(\mathbb E\circ tr)(X^{e_1}\ldots X^{e_p})$$
More generally, assuming that we are given as well a permutation $\sigma\in S_p$, we set
$$M_\sigma^e(X)=\frac{1}{n^{|\sigma|}}\,\mathbb E\left(\sum_{i_1\ldots i_p}X^{e_1}_{i_1i_{\sigma(1)}}\ldots X^{e_p}_{i_pi_{\sigma(p)}}\right)$$
where $n$ is the size of $X$, and $|\sigma|$ is the number of cycles of $\sigma$, and call these numbers generalized $*$-moments of $X$.
\end{definition}

Observe that when the sequence of exponents is $e=(1,\ldots,1)$, the $*$-moment $M^e(X)$ coincides with the usual moment $M_p(X)=(\mathbb E\circ tr)(X^p)$. As already mentioned, the need for the $*$-moments comes from the fact that $X$ can be not self-adjoint. See \cite{vdn}.

Regarding now the generalized $*$-moments, this is a technical definition, in the spirit of \cite{nsp}, that we will need in what follows. Consider the standard cycle in $S_p$, namely:
$$\gamma=(1\to2\to\ldots\to p\to 1)$$

If we use this permutation, the corresponding generalized $*$-moment is:
$$M_\gamma^e(X)=\frac{1}{n}\,\mathbb E\left(\sum_{i_1\ldots i_p}X^{e_1}_{i_1i_2}\ldots X^{e_p}_{i_pi_1}\right)=(\mathbb E\circ tr)(X^{e_1}\ldots X^{e_p})$$

In general, we can decompose the computation of $M_\sigma^e(X)$ over the cycles of $\sigma$, and we obtain in this way a certain product of $*$-moments of $X$. See \cite{nsp}.

Let us go back now to our block-modified Wishart matrix $\tilde{W}=(id\otimes\varphi)W$. According to \cite{bn2}, we can expect the formula for the asymptotic $*$-moments of $\tilde{W}$ to involve in fact not $\varphi$ itself, but rather the $*$-moments of a certain matrix $\Lambda$ associated to $\varphi$. 

To be more precise, the definition that we will need is as follows:

\begin{definition}
We use the Choi-Jamiolkowski correspondence between linear maps $\varphi:M_n(\mathbb C)\to M_n(\mathbb C)$ and square matrices $\Lambda\in M_n(\mathbb C)\otimes M_n(\mathbb C)$ given by
$$\Lambda_{ab,cd}=\varphi(e_{ac})_{bd}$$
with the convention $\Lambda=\sum_{abcd}\Lambda_{ab,cd}e_{ac}\otimes e_{bd}$, where $e_{ab}\in M_n(\mathbb C)$ are the standard generators, given by $e_{ab}:e_b\to e_a$, with $\{e_1,\ldots,e_n\}$ being the standard basis of $\mathbb C^n$.
\end{definition}

In what follows we will use either the map $\varphi$, or the matrix $\Lambda$, depending on what suits us best. Observe that in matrix notation, the entries of $\tilde{W}$ are given by:
$$\tilde{W}_{ia,jb}=\sum_{cd}W_{ic,jd}\varphi(e_{cd})_{ab}=\sum_{cd}\Lambda_{ca,db}W_{ic,jd}$$

Regarding the generalized $*$-moments, we have the following formula, valid for any permutations $\sigma,\tau\in S_p$, that we will heavily use in what follows:
$$(M_\sigma^e\otimes M_\tau^e)(\Lambda)=\frac{1}{n^{|\sigma|+|\tau|}}\sum_{i_1\ldots i_p}\sum_{j_1\ldots j_p}\Lambda_{i_1j_1,i_{\sigma(1)}j_{\tau(1)}}^{e_1}\ldots\ldots\Lambda_{i_pj_p,i_{\sigma(p)}j_{\tau(p)}}^{e_p}$$

Consider the embedding $NC_p\subset S_p$ obtained by ``cycling inside each block''. That is, each block $b=\{b_1,\ldots,b_k\}$ with $b_1<\ldots<b_k$ of a given noncrossing partition $\sigma\in NC_p$ produces by definition the cycle $(b_1\ldots b_k)$ of the corresponding permutation $\sigma\in S_p$. 

Observe that the one-block partition $\gamma\in NC_p$ corresponds in this way to the standard cycle $\gamma\in S_p$. Also, the number of blocks $|\sigma|$ of a partition $\sigma\in NC_p$ corresponds in this way to the number of cycles $|\sigma|$ of the corresponding permutation $\sigma\in S_p$. 

We have now all the needed ingredients for doing our $*$-moment computation. The result here, extending the usual moment computation from \cite{bn2}, is as follows:

\begin{proposition}
The asymptotic $*$-moments of an arbitrary block-modified Wishart matrix $\tilde{W}=(id\otimes\varphi)W$, with parameters $d,m,n\in\mathbb N$, are given by
$$\lim_{d\to\infty}M_p^e\left(m\tilde{W}\right)=\sum_{\sigma\in NC_p}(mn)^{|\sigma|}(M_\sigma^e\otimes M_\gamma^e)(\Lambda)$$
where $\Lambda\in M_n(\mathbb C)\otimes M_n(\mathbb C)$ is the square matrix associated to $\varphi:M_n(\mathbb C)\to M_n(\mathbb C)$.
\end{proposition}

\begin{proof}
We use the formula for the matrix entries of $\tilde{W}$, given after Definition 1.3 above. As a first observation, the entries of the adjoint matrix $\tilde{W}^*$ are given by:
$$\tilde{W}_{ia,jb}^*=\sum_{cd}\bar{\Lambda}_{db,ca}\bar{W}_{jd,ic}=\sum_{cd}\Lambda^*_{ca,db}W_{ic,jd}$$

Thus, we have the following global formula, valid for any exponent $e\in\{1,*\}$:
$$\tilde{W}_{ia,jb}^e=\sum_{cd}\Lambda^e_{ca,db}W_{ic,jd}$$

In order to compute the $*$-moments of $\tilde{W}$, observe first that we have:
\begin{eqnarray*}
tr(\tilde{W}^{e_1}\ldots\tilde{W}^{e_p})
&=&(dn)^{-1}\sum_{i_ra_r}\prod_s\tilde{W}_{i_sa_s,i_{s+1}a_{s+1}}^{e_s}\\
&=&(dn)^{-1}\sum_{i_ra_rc_rd_r}\prod_s\Lambda_{c_sa_s,d_sa_{s+1}}^{e_s}W_{i_sc_s,i_{s+1}d_s}\\
&=&(dn)^{-1}(dm)^{-p}\sum_{i_ra_rc_rd_rj_rb_r}\prod_s\Lambda_{c_sa_s,d_sa_{s+1}}^{e_s}G_{i_sc_s,j_sb_s}\bar{G}_{i_{s+1}d_s,j_sb_s}
\end{eqnarray*}

The average of the general term can be computed by the Wick rule:
$$\mathbb E\left(\prod_sG_{i_sc_s,j_sb_s}\bar{G}_{i_{s+1}d_s,j_sb_s}\right)
=\#\left\{\sigma\in S_p\Big|i_{\sigma(s)}=i_{s+1},c_{\sigma(s)}=d_s,j_{\sigma(s)}=j_s,b_{\sigma(s)}=b_s\right\}$$

Let us look now at the above sum. The $i,j,b$ indices range over sets having respectively $d,d,m$ elements, and they have to be constant under the action of $\sigma\gamma^{-1},\sigma,\sigma$. Thus when summing over these $i,j,b$ indices we simply obtain a $d^{|\sigma\gamma^{-1}|}d^{|\sigma|}m^{|\sigma|}$ factor, so we get:
\begin{eqnarray*}
(\mathbb E\circ tr)(\tilde{W}^{e_1}\ldots\tilde{W}^{e_p})
&=&(dn)^{-1}(dm)^{-p}\sum_{\sigma\in S_p}d^{|\sigma\gamma^{-1}|}(dm)^{|\sigma|}\sum_{a_rc_r}\prod_s\Lambda_{c_sa_s,c_{\sigma(s)}a_{s+1}}^{e_s}\\
&=&m^{-p}\sum_{\sigma\in S_p}d^{|\sigma|+|\sigma\gamma^{-1}|-p-1}(mn)^{|\sigma|}(M_\sigma^e\otimes M_\gamma^e)(\Lambda)
\end{eqnarray*}

We use now the standard fact, from \cite{bia}, that for $\sigma\in S_p$ we have $|\sigma|+|\sigma\gamma^{-1}|\leq p+1$, with equality precisely when $\sigma\in NC_p$. Thus with $d\to\infty$ the sum restricts over the partitions $\sigma\in NC_p$, and this gives the formula in the statement.
\end{proof}

In order to exploit the above formula, we will need some basic facts from free probability theory, namely the general formalism of the $*$-distributions, the free convolution operation $\boxplus$, and the notion of free cumulant. We refer to \cite{nsp}, \cite{vdn} for this material.

We will need in particular the following well-known result:

\begin{proposition}
Given a positive measure $\mu$ on $\mathbb C$, having mass $c>0$, we can define a $*$-distribution by the following Poisson convergence type formula:
$$\pi_\mu=\lim_{n\to\infty}\left(\left(1-\frac{c}{n}\right)\delta_0+\frac{1}{n}\mu\right)^{\boxplus n}$$
This $*$-distribution, called compound free Poisson law, has as free $*$-cumulants the $*$-moments of $\mu$. Moreover, when $\mu=\sum_ic_i\delta_{z_i}$ with $c_i>0$ and $z_i\in\mathbb C$, we have
$$\pi_\mu={\rm law}\left(\sum_iz_i\alpha_i\right)$$
where the variables $\alpha_i$ are free Poisson of parameter $c_i$, taken to be free.
\end{proposition}

\begin{proof}
All these assertions are well-known in the real case, and explained in detail in \cite{bn2}, and the proof in the general complex case is similar. To be more precise, all the assertions follow by using Speicher's moment-cumulant formula. See \cite{nsp}, \cite{spe}.
\end{proof}

Now back to the Wishart matrices, the trick, from \cite{bn2}, is that of constructing a compound free Poisson law, by using the law of $\Lambda$. To be more precise, we have:

\begin{proposition}
Given a square matrix $\Lambda\in M_n(\mathbb C)\otimes M_n(\mathbb C)$, having $*$-distribution $\rho=law(\Lambda)$, the $*$-moments of the compound free Poisson law $\pi_{mn\rho}$ are given by
$$M_p^e(\pi_{mn\rho})=\sum_{\sigma\in NC_p}(mn)^{|\sigma|}(M_\sigma^e\otimes M_\sigma^e)(\Lambda)$$
for any choice of the extra parameter $m\in\mathbb N$.
\end{proposition}

\begin{proof}
We know from Proposition 1.5 above that the free $*$-cumulants of $\pi_{mn\rho}$ are the $*$-moments of $mn\rho$. Thus, these free $*$-cumulants are given by:
$$\kappa_p^e(\pi_{mn\rho})=M_p^e(mn\rho)=mn\cdot M_p^e(\Lambda)=mn\cdot (M_\gamma^e\otimes M_\gamma^e)(\Lambda)$$

By using now Speicher's moment-cumulant formula \cite{nsp}, this gives the result.
\end{proof}

We can see now an obvious similarity with the formula in Proposition 1.4. In order to exploit this similarity, once again by following \cite{bn2}, let us introduce:

\begin{definition}
We call a square matrix $\Lambda\in M_n(\mathbb C)\otimes M_n(\mathbb C)$ multiplicative when
$$(M_\sigma^e\otimes M_\gamma^e)(\Lambda)=(M_\sigma^e\otimes M_\sigma^e)(\Lambda)$$
holds for any $p\in\mathbb N$, any exponents $e_1,\ldots,e_p\in\{1,*\}$, and any $\sigma\in NC_p$.
\end{definition}

With this above notion in hand, we can now formulate an asymptotic $*$-distribution result regarding the block-modified Wishart matrices, as follows:

\begin{theorem}
Consider a block-modified Wishart matrix $\tilde{W}=(id\otimes\varphi)W$, and assume that the matrix $\Lambda\in M_n(\mathbb C)\otimes M_n(\mathbb C)$ associated to $\varphi$ is multiplicative. Then
$$m\tilde{W}\sim\pi_{mn\rho}$$
holds, in $*$-moments, in the $d\to\infty$ limit, where $\rho=law(\Lambda)$.
\end{theorem}

\begin{proof}
By comparing the $*$-moment formulae in Proposition 1.4 and in Proposition 1.6, we conclude that the asymptotic formula $m\tilde{W}\sim\pi_{mn\rho}$ is equivalent to the following equality, which should hold for any $p\in\mathbb N$, and any exponents $e_1,\ldots,e_p\in\{1,*\}$:
$$\sum_{\sigma\in NC_p}(mn)^{|\sigma|}(M_\sigma^e\otimes M_\gamma^e)(\Lambda)=\sum_{\sigma\in NC_p}(mn)^{|\sigma|}(M_\sigma^e\otimes M_\sigma^e)(\Lambda)$$

Now by assuming that $\Lambda$ is multiplicative, in the sense of Definition 1.7 above, these two sums are trivially equal, and this gives the result.
\end{proof}

Summarizing, we have now $*$-moment extensions of the basic results from \cite{bn2}. As explained in \cite{bn2}, and then in \cite{anv}, it is possible to use weaker versions of the multiplicativity condition in Definition 1.7, in order to obtain some more specialized results, in the spirit of Theorem 1.8. In what follows we will use Theorem 1.8 as it is.

\section{Easiness, examples}

In this section and in the next ones we work out some explicit consequences of Theorem 1.8, by using some special classes of modification maps $\varphi:M_n(\mathbb C)\to M_n(\mathbb C)$.

Let us begin with the following standard definition:

\begin{definition}
Let $P(k,l)$ be the set of partitions between an upper row of $k$ points, and a lower row of $l$ points. Associated to any $\pi\in P(k,l)$ is the linear map
$$T_\pi(e_{i_1}\otimes\ldots\otimes e_{i_k})=\sum_{j_1\ldots j_l}\delta_\pi\begin{pmatrix}i_1&\ldots&i_k\\ j_1&\ldots&j_l\end{pmatrix}e_{j_1}\otimes\ldots\otimes e_{j_l}$$
between tensor powers of $\mathbb C^N$, called ``easy'', with the Kronecker type symbol on the right being given by $\delta_\pi=1$ when the indices fit, and $\delta_\pi=0$ otherwise.
\end{definition}

Here $e_1,\ldots,e_N$ is as usual the standard basis of $\mathbb C^N$, with $N\in\mathbb N$ being arbitrary, and the convention is that the indices fit when any block of $\pi$ contains equal indices.

The above maps are well-known in representation theory, the result being that if we denote by $u$ the fundamental representation of the symmetric group $S_N$, we have:
$$Hom(u^{\otimes k},u^{\otimes l})=span\left(T_\pi\Big|\pi\in P(k,l)\right)$$

In what follows we will only need easy maps coming from partitions having even blocks. The representation theory result here, that we will not really need, is as follows:

\begin{proposition}
Consider the hyperoctahedral group $H_N=\mathbb Z_2\wr S_N$, with fundamental representation $v$, coming from the standard action $H_N\curvearrowright \mathbb R^N$. We have then
$$Hom(v^{\otimes k},v^{\otimes l})=span\left(T_\pi\Big|\pi\in P_{even}(k,l)\right)$$
where $P_{even}(k,l)\subset P(k,l)$ is the subset of partitions having blocks of even size.
\end{proposition}

\begin{proof}
Since we have an inclusion $S_N\subset H_N$, when looking at the corresponding invariants we obtain inclusions $Hom(v^{\otimes k},v^{\otimes l})\subset Hom(u^{\otimes k},u^{\otimes l})$, for any $k,l\in\mathbb N$. With this observation in hand, the result follows from the result for $S_N$, the idea being that the sign switches coming from $\mathbb Z_2$ must ``group together'', and so ultimately correspond to restricting the attention to the partitions belonging to $P_{even}(k,l)\subset P(k,l)$. See \cite{bbc}.
\end{proof}

As a conclusion to all this, the general idea is that, at least for certain representation-theoretic purposes, the maps $T_\pi$ from Definition 2.1 above are indeed ``easy''. For full details regarding the notion of easiness, in this setting, we refer to \cite{bsp}, \cite{bra}, \cite{rwe}, \cite{twe}.

Now back to our questions, the idea here, which goes back to the work in \cite{bn2}, is that the same conclusion applies to the block-modified Wishart problematics. To be more precise, the ``easy'' maps for this theory are, once again, those in Definition 2.1 above.

In order to explain this phenomenon, let us begin with:

\begin{definition}
Associated to any partition $\pi\in P(2s,2s)$ is the linear map
$$\varphi_\pi(e_{a_1\ldots a_s,c_1\ldots c_s})=\sum_{b_1\ldots b_s}\sum_{d_1\ldots d_s}\delta_\pi\begin{pmatrix}a_1&\ldots&a_s&c_1&\ldots&c_s\\ b_1&\ldots&b_s&d_1&\ldots&d_s\end{pmatrix}e_{b_1\ldots b_s,d_1\ldots d_s}$$
obtained from $T_\pi$ by contracting all the tensors, via $e_{i_1}\otimes\ldots\otimes e_{i_{2s}}\to e_{i_1\ldots i_s,i_{s+1}\ldots i_{2s}}$.
\end{definition}

Here, as in Definition 2.1 above, $\{e_1,\ldots,e_N\}$ is the standard basis of $\mathbb C^N$, with $N\in\mathbb N$ being some fixed integer, and $\{e_{ij}\}$ is the corresponding basis of $M_N(\mathbb C)$. Thus, the above linear map $\varphi_\pi$ can be viewed as a ``block-modification'' map, as follows:
$$\varphi_\pi:M_{N^s}(\mathbb C)\to M_{N^s}(\mathbb C)$$

In order to verify that the corresponding matrices $\Lambda_\pi$ are multiplicative, we will need to check that all the functions $\varphi(\sigma,\tau)=(M_\sigma^e\otimes M_\tau^e)(\Lambda_\pi)$ have the property $\varphi(\sigma,\gamma)=\varphi(\sigma,\sigma)$. For this purpose, we can use the following result, coming from \cite{bn2}:

\begin{proposition}
The following functions $\varphi:NC_p\times NC_p\to\mathbb R$ are ``multiplicative'', in the sense that they satisfy the condition $\varphi(\sigma,\gamma)=\varphi(\sigma,\sigma)$:
\begin{enumerate}
\item $\varphi(\sigma,\tau)=|\sigma\tau^{-1}|-|\tau|$.

\item $\varphi(\sigma,\tau)=|\sigma\tau|-|\tau|$.

\item $\varphi(\sigma,\tau)=|\sigma\wedge\tau|-|\tau|$.
\end{enumerate}
\end{proposition}

\begin{proof}
These results follow indeed from the following computations:
\begin{eqnarray*}
\varphi_1(\sigma,\gamma)&=&|\sigma\gamma^{-1}|-1=p-|\sigma|=\varphi_1(\sigma,\sigma)\\
\varphi_2(\sigma,\gamma)&=&|\sigma\gamma|-1=|\sigma^2|-|\sigma|=\varphi_2(\sigma,\sigma)\\
\varphi_3(\sigma,\gamma)&=&|\gamma|-|\gamma|=0=|\sigma|-|\sigma|=\varphi_3(\sigma,\sigma)
\end{eqnarray*}

To be more precise, here we have used the formula in \cite{bia} at (1), the formula used in (2) is non-trivial, and was established in \cite{bn2}, and the computation (3) is trivial.
\end{proof}

Let us first discuss the case $s=1$. There are 15 partitions $\pi\in P(2,2)$, and among them, the most ``basic'' ones are the 4 partitions $\pi\in P_{even}(2,2)$. With the convention that $A^\delta\in M_N(\mathbb C)$ denotes the diagonal of a matrix $A\in M_N(\mathbb C)$, we have:

\begin{proposition}
The partitions $\pi\in P_{even}(2,2)$ are as follows,
$$\pi_1=\begin{bmatrix}\circ&\bullet\\ \circ&\bullet\end{bmatrix}\quad,\quad
\pi_2=\begin{bmatrix}\circ&\bullet\\ \bullet&\circ\end{bmatrix}\quad,\quad
\pi_3=\begin{bmatrix}\circ&\circ\\ \bullet&\bullet\end{bmatrix}\quad,\quad
\pi_4=\begin{bmatrix}\circ&\circ\\ \circ&\circ\end{bmatrix}$$
with the associated linear maps $\varphi_\pi:M_n(\mathbb C)\to M_N(\mathbb C)$ being as follows:
$$\varphi_1(A)=A\quad,\quad
\varphi_2(A)=A^t\quad,\quad
\varphi_3(A)=Tr(A)1\quad,\quad
\varphi_4(A)=A^\delta$$
The corresponding matrices $\Lambda_\pi$ are all multiplicative, in the sense of Definition 2.3.
\end{proposition}

\begin{proof}
According to the formula in Definition 2.3, taken at $s=1$, we have:
$$\varphi_\pi(e_{ac})=\sum_{bd}\delta_\pi\begin{pmatrix}a&c\\ b&d\end{pmatrix}e_{bd}$$

In the case of the 4 partitions in the statement, these maps are given by:
$$\varphi_1(e_{ac})=e_{ac}\quad,\quad
\varphi_2(e_{ac})=e_{ca}\quad,\quad
\varphi_3(e_{ac})=\delta_{ac}\sum_be_{bb}\quad,\quad
\varphi_4(e_{ac})=\delta_{ac}e_{aa}$$

Thus, we obtain the formulae in the statement. Regarding now the associated square matrices, appearing via $\Lambda_{ab,cd}=\varphi(e_{ac})_{bd}$, as in Definition 1.3, these are given by:
$$\Lambda^1_{ab,cd}=\delta_{ab}\delta_{cd}\quad,\quad
\Lambda^2_{ab,cd}=\delta_{ad}\delta_{bc}\quad,\quad
\Lambda^3_{ab,cd}=\delta_{ac}\delta_{bd}\quad,\quad
\Lambda^4_{ab,cd}=\delta_{abcd}$$

Since these matrices are all self-adjoint, we can assume that all the exponents are 1 in Definition 1.7, and the condition there becomes $(M_\sigma\otimes M_\gamma)(\Lambda)=(M_\sigma\otimes M_\sigma)(\Lambda)$. In order to check this condition, observe that in the case of the above 4 matrices, we have:
\begin{eqnarray*}
(M_\sigma\otimes M_\tau)(\Lambda_1)&=&\frac{1}{n^{|\sigma|+|\tau|}}\sum_{i_1\ldots i_p}\delta_{i_{\sigma(1)}i_{\tau(1)}}\ldots\delta_{i_{\sigma(p)}i_{\tau(p)}}=n^{|\sigma\tau^{-1}|-|\sigma|-|\tau|}\\
(M_\sigma\otimes M_\tau)(\Lambda_2)&=&\frac{1}{n^{|\sigma|+|\tau|}}\sum_{i_1\ldots i_p}\delta_{i_1i_{\sigma\tau(1)}}\ldots\delta_{i_pi_{\sigma\tau(p)}}=n^{|\sigma\tau|-|\sigma|-|\tau|}\\
(M_\sigma\otimes M_\tau)(\Lambda_3)&=&\frac{1}{n^{|\sigma|+|\tau|}}\sum_{i_1\ldots i_p}\sum_{j_1\ldots j_p}\delta_{i_1i_{\sigma(1)}}\delta_{j_1j_{\tau(1)}}\ldots\delta_{i_pi_{\sigma(p)}}\delta_{j_pj_{\tau(p)}}=1\\
(M_\sigma\otimes M_\tau)(\Lambda_4)&=&\frac{1}{n^{|\sigma|+|\tau|}}\sum_{i_1\ldots i_p}\delta_{i_1i_{\sigma(1)}i_{\tau(1)}}\ldots\delta_{i_pi_{\sigma(p)}i_{\tau(p)}}=n^{|\sigma\wedge\tau|-|\sigma|-|\tau|}
\end{eqnarray*}

By using now the results in Proposition 2.4 above, this gives the result.
\end{proof}

Summarizing, the partitions $\pi\in P_{even}(2,2)$ provide us with some concrete ``input'' for Theorem 1.8. The point now is that, when using this input, we obtain precisely the main known computations for the block-modified Wishart matrices, from \cite{aub}, \cite{cn1}, \cite{cn2}, \cite{mpa}:

\begin{theorem}
The asymptotic distribution results for the block-modified Wishart matrices coming from the partitions $\pi_1,\pi_2,\pi_3,\pi_4\in P_{even}(2,2)$ are as follows:
\begin{enumerate}
\item Marchenko-Pastur: $tW\sim\pi_t$, where $t=m/n$.

\item Aubrun type: $m(id\otimes t)W\sim law(\alpha_+-\alpha_-)$, with $\alpha_\pm\sim\pi_{m(n\pm 1)/2}$, free. 

\item Collins-Nechita one: $t(id\otimes tr(.)1)W\sim\pi_t$, where $t=mn$.

\item Collins-Nechita two: $m(id\otimes(.)^\delta)W\sim\pi_m$.
\end{enumerate}
\end{theorem}

\begin{proof}
These observations go back to \cite{bn2}. In our setting, the maps $\varphi_1,\varphi_2,\varphi_3,\varphi_4$ in Proposition 2.5 give the 4 matrices in the statement, modulo some rescalings, and the computation of the corresponding distributions, using Theorem 1.8, goes as follows:

(1) Here $\Lambda=\sum_{ac}e_{ac}\otimes e_{ac}$, and so $\Lambda=nP$, where $P$ is the rank one projection on the vector $\sum_ae_a\otimes e_a\in\mathbb C^n\otimes\mathbb C^n$. Thus $\rho=\frac{n^2-1}{n^2}\delta_0+\frac{1}{n^2}\delta_n$, and this gives the result.

(2) Here $\Lambda=\sum_{ac}e_{ac}\otimes e_{ca}$ is the flip operator, $\Lambda(e_c\otimes e_a)=e_a\otimes e_c$. Thus $\rho=\frac{n-1}{2n}\delta_{-1}+\frac{n+1}{2n}\delta_1$, and so $mn\rho=\frac{m(n-1)}{2}\delta_{-1}+\frac{m(n+1)}{2}\delta_1$, which gives the result.

(3) Here $\Lambda=\sum_{ab}e_{aa}\otimes e_{bb}$ is the identity matrix, $\Lambda=1$. Thus in this case we have $\rho=\delta_1$, so $\pi_{mn\rho}=\pi_{mn}$, and so $m\tilde{W}\sim\pi_{mn}$, as claimed.

(4) Here $\Lambda=\sum_ae_{aa}\otimes e_{aa}$ is the orthogonal projection on $span(e_a\otimes e_a)\subset\mathbb C^n\otimes\mathbb C^n$. Thus we have $\rho=\frac{n-1}{n}\delta_0+\frac{1}{n}\delta_1$, and this gives the result.
\end{proof}

Summarizing, in what regards the block-modified Wishart matrix theory, the ``simplest'' modification maps $\varphi$ are the easy maps $\varphi_\pi$ coming from partitions $\pi\in P_{even}(2,2)$. Further exploiting this observation will be our main task, in what follows.

\section{General theory}

We develop now some general theory, for the partitions $\pi\in P_{even}(2s,2s)$, with $s\in\mathbb N$ being arbitrary. Our main result will basically refine the main diagrammatic result from \cite{bn2}, by lifting some assumptions from there, which are in fact not needed.

Let us begin with a reformulation of Definition 2.3, in terms of square matrices:

\begin{proposition}
Given $\pi\in P(2s,2s)$, the square matrix $\Lambda_\pi\in M_n(\mathbb C)\otimes M_n(\mathbb C)$ associated to the linear map $\varphi_\pi:M_n(\mathbb C)\to M_n(\mathbb C)$, with $n=N^s$, is given by:
$$(\Lambda_\pi)_{a_1\ldots a_s,b_1\ldots b_s,c_1\ldots c_s,d_1\ldots d_s}=
\delta_\pi\begin{pmatrix}a_1&\ldots&a_s&c_1&\ldots&c_s\\ b_1&\ldots&b_s&d_1&\ldots&d_s\end{pmatrix}$$
In addition, we have $\Lambda_\pi^*=\Lambda_{\pi^\circ}$, where $\pi\to\pi^\circ$ is the blockwise middle symmetry.
\end{proposition}

\begin{proof}
The formula for $\Lambda_\pi$ follows from the formula of $\varphi_\pi$ from Definition 2.3, by using our standard convention $\Lambda_{ab,cd}=\varphi(e_{ac})_{bd}$. Regarding now the second assertion, observe that with $\pi\to\pi^\circ$ being as above, for any multi-indices $a,b,c,d$ we have:
$$\delta_\pi\begin{pmatrix}c_1&\ldots&c_s&a_1&\ldots&a_s\\ d_1&\ldots&d_s&b_1&\ldots&b_s\end{pmatrix}
=\delta_{\pi^\circ}\begin{pmatrix}a_1&\ldots&a_s&c_1&\ldots&c_s\\ b_1&\ldots&b_s&d_1&\ldots&d_s\end{pmatrix}$$

Since $\Lambda_\pi$ is real, we conclude we have the following formula:
$$(\Lambda_\pi^*)_{ab,cd}=(\Lambda_\pi)_{cd,ab}=(\Lambda_{\pi^\circ})_{ab,cd}$$

This being true for any $a,b,c,d$, we obtain $\Lambda_\pi^*=\Lambda_{\pi^\circ}$, as claimed.
\end{proof}

In order to compute now the generalized $*$-moments of $\Lambda_\pi$, we first have:

\begin{proposition}
With $\pi\in P(2s,2s)$ and $\Lambda_\pi$ being as above, we have
\begin{eqnarray*}
(M_\sigma^e\otimes M_\tau^e)(\Lambda_\pi)
&=&\frac{1}{n^{|\sigma|+|\tau|}}\sum_{i_1^1\ldots i_p^s}\sum_{j_1^1\ldots j_p^s}
\delta_{\pi^{e_1}}\begin{pmatrix}i_1^1&\ldots&i_1^s&i_{\sigma(1)}^1&\ldots&i_{\sigma(1)}^s\\
j_1^1&\ldots&j_1^s&j_{\tau(1)}^1&\ldots&j_{\tau(1)}^s\end{pmatrix}\\
&&\hskip62mm\vdots\\
&&\hskip31mm\delta_{\pi^{e_p}}\begin{pmatrix}i_p^1&\ldots&i_p^s& i_{\sigma(p)}^1&\ldots&i_{\sigma(p)}^s\\
j_p^1&\ldots&j_p^s&j_{\tau(p)}^1&\ldots&j_{\tau(p)}^s\end{pmatrix}
\end{eqnarray*}
with the exponents $e_1,\ldots,e_p\in\{1,*\}$ at left corresponding to $e_1,\ldots,e_p\in\{1,\circ\}$ at right.
\end{proposition}

\begin{proof}
In multi-index notation, the general formula for the generalized $*$-moments for a tensor product square matrix $\Lambda\in M_n(\mathbb C)\otimes M_n(\mathbb C)$, with $n=N^s$, is:
\begin{eqnarray*}
(M_\sigma^e\otimes M_\tau^e)(\Lambda)
&=&\frac{1}{n^{|\sigma|+|\tau|}}\sum_{i_1^1\ldots i_p^s}\sum_{j_1^1\ldots j_p^s}
\Lambda^{e_1}_{i_1^1\ldots i_1^sj_1^1\ldots j_1^s,i_{\sigma(1)}^1\ldots i_{\sigma(1)}^sj_{\tau(1)}^1\ldots j_{\tau(1)}^s}\\
&&\hskip52mm\vdots\\
&&\hskip30mm\Lambda^{e_p}_{i_p^1\ldots i_p^sj_p^1\ldots j_p^s,i_{\sigma(p)}^1\ldots i_{\sigma(p)}^sj_{\tau(p)}^1\ldots j_{\tau(p)}^s}
\end{eqnarray*}

By using now the formulae in Proposition 3.1 above for the matrix entries of $\Lambda_\pi$, and of its adjoint matrix $\Lambda_\pi^*=\Lambda_{\pi^\circ}$, this gives the formula in the statement.
\end{proof}

As a conclusion, the quantities $(M_\sigma^e\otimes M_\tau^e)(\Lambda_\pi)$ that we are interested in can be theoretically computed in terms of $\pi$, but the combinatorics is quite non-trivial.

As explained in \cite{bn2}, some simplifications appear in the symmetric case, $\pi=\pi^\circ$. Indeed, for such partitions we can use the following decomposition result:

\begin{proposition}
Each symmetric partition $\pi\in P_{even}(2s,2s)$ has a finest symmetric decomposition $\pi=[\pi_1,\ldots,\pi_R]$, with the components $\pi_t$ being of two types, as follows:
\begin{enumerate}
\item Symmetric blocks of $\pi$. Such a block must have $r+r$ matching upper legs and $v+v$ matching lower legs, with $r+v>0$.

\item Unions $\beta\sqcup\beta^\circ$ of asymmetric blocks of $\pi$. Here $\beta$ must have $r+u$ unmatching upper legs and $v+w$ unmatching lower legs, with $r+u+v+w>0$.
\end{enumerate}
\end{proposition}

\begin{proof}
Consider indeed the block decomposition of our partition, $\pi=[\beta_1,\ldots,\beta_T]$. Then $[\beta_1,\ldots,\beta_T]=[\beta_1^\circ,\ldots,\beta_T^\circ]$, so each block $\beta\in\pi$ is either symmetric, $\beta=\beta^\circ$, or is asymmetric, and disjoint from $\beta^\circ$, which must be a block of $\pi$ too. The result follows.
\end{proof}

The idea will be that of decomposing over the components of $\pi$. First, we have:

\begin{proposition}
For the standard pairing $\eta\in P_{even}(2s,2s)$ having horizontal strings,
$$\eta=\begin{bmatrix}
a&b&c&\ldots&a&b&c&\ldots\\
\alpha&\beta&\gamma&\ldots&\alpha&\beta&\gamma&\ldots
\end{bmatrix}$$
we have $(M_\sigma\otimes M_\tau)(\Lambda_\eta)=1$, for any $p\in\mathbb N$, and any $\sigma,\tau\in NC_p$.
\end{proposition}

\begin{proof}
As a first observation, the result holds indeed at $s=1$, due to the computations in the proof of Proposition 2.5. In general, by using Proposition 3.2, we obtain:
\begin{eqnarray*}
(M_\sigma\otimes M_\tau)(\Lambda_\eta)
&=&\frac{1}{n^{|\sigma|+|\tau|}}\sum_{i_1^1\ldots i_p^s}\sum_{j_1^1\ldots j_p^s}\delta_{i_1^1i_{\sigma(1)}^1}\ldots\delta_{i_1^si_{\sigma(1)}^s}\cdot\delta_{j_1^1j_{\tau(1)}^1}\ldots\delta_{j_1^sj_{\tau(1)}^s}\\
&&\hskip52mm\vdots\\
&&\hskip30mm\delta_{i_p^1i_{\sigma(p)}^1}\ldots\delta_{i_p^si_{\sigma(p)}^s}\cdot\delta_{j_p^1j_{\tau(p)}^1}\ldots\delta_{j_p^sj_{\tau(p)}^s}
\end{eqnarray*}

By transposing the two $p\times s$ matrices of Kronecker symbols, we obtain:
\begin{eqnarray*}
(M_\sigma\otimes M_\tau)(\Lambda_\eta)
&=&\frac{1}{n^{|\sigma|+|\tau|}}\sum_{i_1^1\ldots i_p^1}\sum_{j_1^1\ldots j_p^1}\delta_{i_1^1i_{\sigma(1)}^1}\ldots\delta_{i_p^1i_{\sigma(p)}^1}\cdot\delta_{j_1^1j_{\tau(1)}^1}\ldots\delta_{j_p^1j_{\tau(p)}^1}\\
&&\hskip52mm\vdots\\
&&\hskip13.5mm\sum_{i_1^s\ldots i_p^s}\sum_{j_1^s\ldots j_p^s}\delta_{i_1^si_{\sigma(1)}^s}\ldots\delta_{i_p^si_{\sigma(p)}^s}\cdot\delta_{j_1^sj_{\tau(1)}^s}\ldots\delta_{j_p^sj_{\tau(p)}^s}
\end{eqnarray*}

We can now perform all the sums, and we obtain in this way:
$$(M_\sigma\otimes M_\tau)(\Lambda_\eta)
=\frac{1}{n^{|\sigma|+|\tau|}}(N^{|\sigma|}N^{|\tau|})^s=1$$
 
Thus, the formula in the statement holds indeed.
\end{proof}

We can now perform the decomposition over the components, as follows:

\begin{theorem}
Assuming that $\pi\in P_{even}(2s,2s)$ is symmetric, $\pi=\pi^\circ$, we have
$$(M_\sigma\otimes M_\tau)(\Lambda_\pi)=\prod_{t=1}^R(M_\sigma\otimes M_\tau)(\Lambda_{\pi_t})$$
whenever $\pi=[\pi_1,\ldots,\pi_R]$ is a decomposition into symmetric subpartitions, which each $\pi_t$ being completed with horizontal strings, coming from the standard pairing $\eta$.
\end{theorem}

\begin{proof}
We use the general formula in Proposition 3.2. In the symmetric case the various $e_x$ exponents dissapear, and we can write the formula there as follows:
$$(M_\sigma\otimes M_\tau)(\Lambda_\pi)
=\frac{1}{n^{|\sigma|+|\tau|}}\#\left\{i,j\Big|\ker\begin{pmatrix}i_x^1&\ldots&i_x^s&i_{\sigma(x)}^1&\ldots&i_{\sigma(x)}^s\\
j_x^1&\ldots&j_x^s&j_{\tau(x)}^1&\ldots&j_{\tau(x)}^s\end{pmatrix}\leq\pi,\forall x\right\}$$

The point now is that in this formula, the number of double arrays $[ij]$ that we are counting naturally decomposes over the subpartitions $\pi_t$. Thus, we have a formula of the following type, with $K$ being a certain normalization constant:
$$(M_\sigma\otimes M_\tau)(\Lambda_\pi)=K\prod_{t=1}^R(M_\sigma\otimes M_\tau)(\Lambda_{\pi_t})$$

Regarding now the precise value of $K$, our claim is that this is given by:
$$K=\frac{n^{(|\sigma|+|\tau|)R}}{n^{|\sigma|+|\tau|}}\cdot\frac{1}{n^{(|\sigma|+|\tau|)(R-1)}}=1$$

Indeed, the fraction on the left comes from the standard $\frac{1}{n^{|\sigma|+|\tau|}}$ normalizations of all the $(M_\sigma\otimes M_\tau)(\Lambda)$ quantities involved. As for the term on the right, this comes from the contribution of the horizontal strings, which altogether contribute as the strings of the standard pairing $\eta\in P_{even}(2s,2s)$, counted $R-1$ times. But, according to Proposition 3.4 above, the strings of $\eta$ contribute with a $n^{|\sigma|+|\tau|}$ factor, and this gives the result.
\end{proof}

Summarizing, in the easy case we are led to the study of the partitions $\pi\in P_{even}(2s,2s)$ which are symmetric, and we have so far a decomposition formula for them.

\section{Symmetric components}

In this section we keep building on the material developed above. Our purpose is that of converting Theorem 3.5 into an explicit formula, that we can use later on.

We have to compute the contributions of the components. First, we have:

\begin{proposition}
For a symmetric partition $\pi\in P_{even}(2s,2s)$, consisting of one symmetric block, completed with horizontal strings, we have
$$(M_\sigma\otimes M_\tau)(\Lambda_\pi)=N^{|\lambda|-r|\sigma|-v|\tau|}$$
where $\lambda\in P_p$ is a partition constructed as follows,
$$\lambda=\begin{cases}
\sigma\wedge\tau&{\rm if}\ r,v\geq1\\
\sigma&{\rm if}\ r\geq1,v=0\\
\tau&{\rm if}\ r=0,v\geq1
\end{cases}$$
and where $r/v$ is half of the number of upper/lower legs of the symmetric block.
\end{proposition}

\begin{proof}
Let us denote by $a_1,\ldots,a_r$ and $b_1,\ldots,b_v$ the upper and lower legs of the symmetric block, appearing at left, and by $A_1,\ldots,A_{s-r}$ and $B_1,\ldots,B_{s-v}$ the remaining legs, appearing at left as well. With this convention, the formula in Proposition 3.2 gives:
\begin{eqnarray*}
(M_\sigma\otimes M_\tau)(\Lambda_\pi)
&=&\frac{1}{n^{|\sigma|+|\tau|}}\sum_{i_1^1\ldots i_p^s}\sum_{j_1^1\ldots j_p^s}\prod_x\delta_{i_x^{a_1}\ldots i_x^{a_r}i_{\sigma(x)}^{a_1}\ldots i_{\sigma(x)}^{a_r}j_x^{b_1}\ldots j_x^{b_v}j_{\tau(x)}^{b_1}\ldots j_{\tau(x)}^{b_v}}\\
&&\hskip37mm\delta_{i_x^{A_1}i_{\sigma(x)}^{A_1}}\ldots\ldots\delta_{i_x^{A_{s-r}}i_{\sigma(x)}^{A_{s-r}}}\\
&&\hskip37mm\delta_{j_x^{B_1}j_{\tau(x)}^{B_1}}\ldots\ldots\delta_{j_x^{B_{s-v}}j_{\tau(x)}^{B_{s-v}}}
\end{eqnarray*}

If we denote by $k_1,\ldots,k_p$ the common values of the indices affected by the long Kronecker symbols, coming from the symmetric block, we have then:
\begin{eqnarray*}
(M_\sigma\otimes M_\tau)(\Lambda_\pi)
&=&\frac{1}{n^{|\sigma|+|\tau|}}\sum_{k_1\ldots k_p}\\
&&\sum_{i_1^1\ldots i_p^s}\prod_x\delta_{i_x^{a_1}\ldots i_x^{a_r}i_{\sigma(x)}^{a_1}\ldots i_{\sigma(x)}^{a_r}k_x}\cdot\delta_{i_x^{A_1}i_{\sigma(x)}^{A_1}}\ldots\delta_{i_x^{A_{s-r}}i_{\sigma(x)}^{A_{s-r}}}\\
&&\sum_{j_1^1\ldots j_p^s}\prod_x\delta_{j_x^{b_1}\ldots j_x^{b_v}j_{\tau(x)}^{b_1}\ldots j_{\tau(x)}^{b_v}k_x}\cdot\delta_{j_x^{B_1}j_{\tau(x)}^{B_1}}\ldots\delta_{j_x^{B_{s-v}}j_{\tau(x)}^{B_{s-v}}}
\end{eqnarray*}

Let us compute now the contributions of the various $i,j$ indices involved. If we regard both $i,j$ as being $p\times s$ arrays of indices, the situation is as follows:

-- On the $a_1,\ldots,a_r$ columns of $i$, the equations are $i_x^{a_e}=i_{\sigma(x)}^{a_e}=k_x$ for any $e,x$. Thus when $r\neq0$ we must have $\ker k\leq\sigma$, in order to have solutions, and if this condition is satisfied, the solution is unique. As for the case $r=0$, here there is no special condition to be satisfied by $k$, and we have once again a unique solution.

-- On the $A_1,\ldots,A_{s-r}$ columns of $i$, the conditions on the indices are the ``trivial'' ones, examined in the proof of Proposition 3.4 above. According to the computation there, the total contribution coming from these indices is $(N^{|\sigma|})^{s-r}=N^{(s-r)|\sigma|}$.

-- Regarding now $j$, the situation is similar, with a unique solution coming from the $b_1,\ldots,b_v$ columns, provided that the condition $\ker k\leq\tau$ is satisfied at $v\neq0$, and with a total $N^{(s-v)|\tau|}$ contribution coming from the $B_1,\ldots,B_{s-v}$ columns.

As a conclusion, in order to have solutions $i,j$, we are led to the condition $\ker k\leq\lambda$, where $\lambda\in\{\sigma\wedge\tau,\sigma,\tau\}$ is the partition constructed in the statement.

Now by putting everything together, we deduce that we have:
\begin{eqnarray*}
(M_\sigma\otimes M_\tau)(\Lambda_\pi)
&=&\frac{1}{n^{|\sigma|+|\tau|}}\sum_{\ker k\leq\lambda}N^{(s-r)|\sigma|+(s-v)|\tau|}\\
&=&N^{-s|\sigma|-s|\tau|}N^{|\lambda|}N^{(s-r)|\sigma|+(s-v)|\tau|}\\
&=&N^{|\lambda|-r|\sigma|-v|\tau|}
\end{eqnarray*}

Thus, we have obtained the formula in the statement, and we are done.
\end{proof}

In the two-block case now, we have a similar result, as follows:

\begin{proposition}
For a symmetric partition $\pi\in P_{even}(2s,2s)$, consisting of a symmetric union $\beta\sqcup\beta^\circ$ of two asymmetric blocks, completed with horizontal strings, we have
$$(M_\sigma\otimes M_\tau)(\Lambda_\pi)=N^{|\lambda|-(r+u)|\sigma|-(v+w)|\tau|}$$
where $r+u$ and $v+w$ represent the number of upper and lower legs of $\beta$, and where $\lambda\in P_p$ is a partition constructed according to the following table,
$$\begin{matrix}
ru\backslash vw&&11&10&01&00\\
\\
11&&\sigma^2\wedge\sigma\tau\wedge\sigma\tau^{-1}&\sigma^2\wedge\sigma\tau^{-1}&\sigma^2\wedge\sigma\tau&\sigma^2\\
10&&\sigma\tau\wedge\sigma\tau^{-1}&\sigma\tau^{-1}&\sigma\tau&\emptyset\\
01&&\tau\sigma\wedge\tau^2&\tau\sigma&\tau^{-1}\sigma&\emptyset\\
00&&\tau^2&\emptyset&\emptyset&-
\end{matrix}$$
with the $1/0$ indexing symbols standing for the positivity/nullness of the corresponding variables $r,u,v,w$, and where $\emptyset$ denotes a formal partition, having $0$ blocks.
\end{proposition}

\begin{proof}
Let us denote by $a_1,\ldots,a_r$ and $c_1,\ldots,c_u$ the upper legs of $\beta$, by $b_1,\ldots,b_v$ and $d_1,\ldots,d_w$ the lower legs of $\beta$, and by $A_1,\ldots,A_{s-r-u}$ and $B_1,\ldots,B_{s-v-w}$ the remaining legs of $\pi$, not belonging to $\beta\sqcup\beta^\circ$. The formula in Proposition 3.2 gives:
\begin{eqnarray*}
(M_\sigma\otimes M_\tau)(\Lambda_\pi)
&=&\frac{1}{n^{|\sigma|+|\tau|}}\sum_{i_1^1\ldots i_p^s}\sum_{j_1^1\ldots j_p^s}\prod_x\delta_{i_x^{a_1}\ldots i_x^{a_r}i_{\sigma(x)}^{c_1}\ldots i_{\sigma(x)}^{c_u}j_x^{b_1}\ldots j_x^{b_v}j_{\tau(x)}^{d_1}\ldots j_{\tau(x)}^{d_w}}\\
&&\hskip37mm\delta_{i_x^{c_1}\ldots i_x^{c_u}i_{\sigma(x)}^{a_1}\ldots i_{\sigma(x)}^{a_r}j_x^{d_1}\ldots j_x^{d_w}j_{\tau(x)}^{b_1}\ldots j_{\tau(x)}^{b_v}}\\
&&\hskip37mm\delta_{i_x^{A_1}i_{\sigma(x)}^{A_1}}\ldots\ldots\delta_{i_x^{A_{s-r}}i_{\sigma(x)}^{A_{s-r-u}}}\\
&&\hskip37mm\delta_{j_x^{B_1}j_{\tau(x)}^{B_1}}\ldots\ldots\delta_{j_x^{B_{s-v}}j_{\tau(x)}^{B_{s-v-w}}}
\end{eqnarray*}

We have now two long Kronecker symbols, coming from $\beta\sqcup\beta^\circ$, and if we denote by $k_1,\ldots,k_p$ and $l_1,\ldots,l_p$ the values of the indices affected by them, we obtain:
\begin{eqnarray*}
&&(M_\sigma\otimes M_\tau)(\Lambda_\pi)
=\frac{1}{n^{|\sigma|+|\tau|}}\sum_{k_1\ldots k_p}\sum_{l_1\ldots l_p}\\
&&\hskip20mm\sum_{i_1^1\ldots i_p^s}\prod_x\delta_{i_x^{a_1}\ldots i_x^{a_r}i_{\sigma(x)}^{c_1}\ldots i_{\sigma(x)}^{c_u}k_x}\cdot\delta_{i_x^{c_1}\ldots i_x^{c_u}i_{\sigma(x)}^{a_1}\ldots i_{\sigma(x)}^{a_r}l_x}\cdot\delta_{i_x^{A_1}i_{\sigma(x)}^{A_1}}\ldots\delta_{i_x^{A_{s-r-u}}i_{\sigma(x)}^{A_{s-r-u}}}\\
&&\hskip20mm\sum_{j_1^1\ldots j_p^s}\prod_x\delta_{j_x^{b_1}\ldots j_x^{b_v}j_{\tau(x)}^{d_1}\ldots j_{\tau(x)}^{d_w}k_x}\cdot\delta_{j_x^{d_1}\ldots j_x^{d_w}j_{\tau(x)}^{b_1}\ldots j_{\tau(x)}^{b_v}l_x}\cdot\delta_{j_x^{B_1}j_{\tau(x)}^{B_1}}\ldots\delta_{j_x^{B_{s-v-w}}j_{\tau(x)}^{B_{s-v-w}}}
\end{eqnarray*}

Let us compute now the contributions of the various $i,j$ indices. On the $a_1,\ldots,a_r$ and $c_1,\ldots,c_u$ columns of $i$, regarded as an $p\times s$ array, the equations are as follows:
$$i_x^{a_e}=i_{\sigma(x)}^{c_f}=k_x\quad,\quad i_x^{c_f}=i_{\sigma(x)}^{a_e}=l_x$$

If we denote by $i_x$ the common value of the $i_x^{a_e}$ indices, when $e$ varies, and by $I_x$ the common value of the $i_x^{c_f}$ indices, when $f$ varies, these equations simply become:
$$i_x=I_{\sigma(x)}=k_x\quad,\quad I_x=i_{\sigma(x)}=l_x$$

Thus we have 0 or 1 solutions. To be more precise, depending now on the positivity/nullness of the parameters $r,u$, we are led to 4 cases, as follows:

\underline{Case 11.} Here $r,u\geq1$, and we must have $k_x=l_{\sigma(x)},k_{\sigma(x)}=l_x$.

\underline{Case 10.} Here $r\geq1,u=0$, and we must have $k_{\sigma(x)}=l_x$.

\underline{Case 01.} Here $r=0,u\geq1$, and we must have $k_x=l_{\sigma(x)}$.

\underline{Case 00.} Here $r=u=0$, and there is no condition on $k,l$.

In what regards now the $A_1,\ldots,A_{s-r}$ columns of $i$, the conditions on the indices are the ``trivial'' ones, examined in the proof of Proposition 3.4 above. According to the computation there, the total contribution coming from these indices is:
$$C_i=(N^{|\sigma|})^{s-r}=N^{(s-r)|\sigma|}$$

The study for the $j$ indices is similar, and we will only record here the final conclusions. First, in what regards the $b_1,\ldots,b_v$ and $d_1,\ldots,d_w$ columns of $j$, the same discussion as above applies, and we have once again 0 or 1 solutions, as follows:

\underline{Case 11'.} Here $v,w\geq1$, and we must have $k_x=l_{\tau(x)},k_{\tau(x)}=l_x$.

\underline{Case 10'.} Here $v\geq1,w=0$, and we must have $k_{\tau(x)}=l_x$.

\underline{Case 01'.} Here $v=0,w\geq1$, and we must have $k_x=l_{\tau(x)}$.

\underline{Case 00'.} Here $v=w=0$, and there is no condition on $k,l$.

As for the $B_1,\ldots,B_{s-v-w}$ columns of $j$, the conditions on the  indices here are ``trivial'', as in Proposition 3.4, and the total contribution coming from these indices is:
$$C_j=(N^{|\tau|})^{s-v-w}=N^{(s-v-w)|\tau|}$$

Let us put now everything together. First, we must merge the conditions on $k,l$ found in the cases 00-11 above with those found in the cases 00'-11'. There are $4\times4=16$ computations to be performed here, and the ``generic'' computation, corresponding to the merger of case 11 with the case 11', is as follows:
\begin{eqnarray*}
&&k_x=l_{\sigma(x)},k_{\sigma(x)}=l_x,k_x=l_{\tau(x)},k_{\tau(x)}=l_x\\
&\iff&l_x=k_{\sigma(x)},k_x=l_{\sigma(x)},k_x=l_{\tau(x)},k_x=l_{\tau^{-1}(x)}\\
&\iff&l_x=k_{\sigma(x)},k_x=k_{\sigma^2(x)}=k_{\sigma\tau(x)}=k_{\sigma\tau^{-1}(x)}
\end{eqnarray*}

Thus in this case $l$ is uniquely determined by $k$, and $k$ itself must satisfy:
$$\ker k\leq\sigma^2\wedge\sigma\tau\wedge\sigma\tau^{-1}$$

We conclude that the total contribution of the $k,l$ indices in this case is:
$$C_{kl}^{11,11}=N^{|\sigma^2\wedge\sigma\tau\wedge\sigma\tau^{-1}|}$$

In the remaining 15 cases the computations are similar, with some of the above 4 conditions, that we started with, dissapearing. The conclusion is that the total contribution of the $k,l$ indices is as follows, with $\lambda$ being the partition in the statement:
$$C_{kl}=N^{|\lambda|}$$

With this result in hand, we can now finish our computation, as follows:
\begin{eqnarray*}
(M_\sigma\otimes M_\tau)(\Lambda_\pi)
&=&\frac{1}{n^{|\sigma|+|\tau|}}C_{kl}C_iC_j\\
&=&N^{|\lambda|-(r+u)|\sigma|-(v+w)|\tau|}
\end{eqnarray*}

Thus, we have obtained the formula in the statement, and we are done.
\end{proof}

As a conclusion now, we have the following result:

\begin{theorem}
For a symmetric partition $\pi\in P_{even}(2s,2s)$, having only one component, in the sense of Proposition 3.3, completed with horizontal strings, we have
$$(M_\sigma\otimes M_\tau)(\Lambda_\pi)=N^{|\lambda|-r|\sigma|-v|\tau|}$$
where $\lambda\in P_p$ is the partition constructed as in Proposition 4.1 and Proposition 4.2, and where $r/v$ is half of the total number of upper/lower legs of the component.
\end{theorem}

\begin{proof}
This follows indeed from Proposition 4.1 and Proposition 4.2 above.
\end{proof}

\section{Exclusion work}

Generally speaking, the formula in Theorem 4.3 does not lead to the multiplicativity condition from Definition 1.7, and this due to the fact that the various partitions $\lambda\in P_p$ constructed in Proposition 4.2 have in general a quite complicated combinatorics. 

To be more precise, we first have the following result:

\begin{proposition}
For a symmetric partition $\pi\in P_{even}(2s,2s)$ we have
$$(M_\sigma\otimes M_\tau)(\Lambda_\pi)=N^{f_1+f_2}$$
where $f_1,f_2$ are respectively linear combinations of the following quantities:
\begin{enumerate}
\item $1,|\sigma|,|\tau|,|\sigma\wedge\tau|,|\sigma\tau|,|\sigma\tau^{-1}|,|\tau\sigma|,|\tau^{-1}\sigma|$.

\item $|\sigma^2|,|\tau^2|,|\sigma^2\wedge\sigma\tau|,|\sigma^2\wedge\sigma\tau^{-1}|,|\tau\sigma\wedge\tau^2|,|\sigma\tau\wedge\sigma\tau^{-1}|,|\sigma^2\wedge\sigma\tau\wedge\sigma\tau^{-1}|$.
\end{enumerate}
\end{proposition}

\begin{proof}
This follows indeed by combining Theorem 3.5 and Theorem 4.3 above, with concrete input from Proposition 4.1 and Proposition 4.2.
\end{proof}

In the above result, the partitions in (1) lead to the multiplicativity condition in Definition 1.7, and so to compound free Poisson laws, via Theorem 1.8. However, the partitions in (2) have a more complicated combinatorics, which does not fit with Definition 1.7 above, nor with the finer multiplicativity notions introduced in \cite{bn2}.

Summarizing, in order to extend the 4 basic computations from section 2, we must fine-tune our formalism. A natural answer here comes from the following result:

\begin{proposition}
For a partition $\pi\in P(2s,2s)$, the following are equivalent:
\begin{enumerate}
\item $\varphi_\pi$ is unital modulo scalars, i.e. $\varphi_\pi(1)=c1$, with $c\in\mathbb C$.

\item $[^\mu_\pi]=\mu$, where $\mu\in P(0,2s)$ is the pairing connecting $\{i\}-\{i+s\}$, and where $[^\mu_\pi]\in P(0,2s)$ is the partition obtained by putting $\mu$ on top of $\pi$.
\end{enumerate}
In addition, these conditions are satisfied for the $4$ partitions in $P_{even}(2,2)$.
\end{proposition}

\begin{proof}
We use the formula of $\varphi_\pi$ from Definition 2.3 above, namely:
$$\varphi_\pi(e_{a_1\ldots a_s,c_1\ldots c_s})=\sum_{b_1\ldots b_s}\sum_{d_1\ldots d_s}\delta_\pi\begin{pmatrix}a_1&\ldots&a_s&c_1&\ldots&c_s\\ b_1&\ldots&b_s&d_1&\ldots&d_s\end{pmatrix}e_{b_1\ldots b_s,d_1\ldots d_s}$$

By summing over indices $a_i=c_i$, we obtain the following formula:
$$\varphi_\pi(1)=\sum_{a_1\ldots a_s}\sum_{b_1\ldots b_s}\sum_{d_1\ldots d_s}
\delta_\pi\begin{pmatrix}a_1&\ldots&a_s&a_1&\ldots&a_s\\ b_1&\ldots&b_s&d_1&\ldots&d_s\end{pmatrix}e_{b_1\ldots b_s,d_1\ldots d_s}$$

Let us first find out when $\varphi_\pi(1)$ is diagonal. In order for this condition to hold, the off-diagonal terms of $\varphi_\pi(1)$ must all vanish, and so we must have:
$$b\neq d\implies\delta_\pi\begin{pmatrix}a_1&\ldots&a_s&a_1&\ldots&a_s\\ b_1&\ldots&b_s&d_1&\ldots&d_s\end{pmatrix}=0,\forall a$$

Our claim is that for any $\pi\in P(2s,2s)$ we have the following formula:
$$\sup_{a_1\ldots a_s}\delta_\pi\begin{pmatrix}a_1&\ldots&a_s&a_1&\ldots&a_s\\ b_1&\ldots&b_s&d_1&\ldots&d_s\end{pmatrix}=\delta_{[^\mu_\pi]}\begin{pmatrix}b_1&\ldots&b_s&d_1&\ldots&d_s\end{pmatrix}$$

Indeed, each of the terms of the sup on the left are smaller than the quantity on the right, so $\leq$ holds. Also, assuming $\delta_{[^\mu_\pi]}(bd)=1$, we can take $a_1,\ldots,a_s$ to be the indices appearing on the strings of $\mu$, and we obtain $\delta_\pi(^a_b{\ }^a_d)=1$. Thus, we have equality.

Now with this equality in hand, we conclude that we have:
\begin{eqnarray*}
&&\varphi_\pi(1)=\varphi_\pi(1)^\delta\\
&\iff&\delta_{[^\mu_\pi]}\begin{pmatrix}b_1&\ldots&b_s&d_1&\ldots&d_s\end{pmatrix}=0,\forall b\neq d\\
&\iff&\delta_{[^\mu_\pi]}\begin{pmatrix}b_1&\ldots&b_s&d_1&\ldots&d_s\end{pmatrix}\leq\delta_\mu\begin{pmatrix}b_1&\ldots&b_s&d_1&\ldots&d_s\end{pmatrix},\forall b,d\\
&\iff&\begin{bmatrix}\mu\\ \pi\end{bmatrix}\leq\mu
\end{eqnarray*}

Let us investigate now when (1) holds. We already know that $\pi$ must satisfy $[^\mu_\pi]\leq\mu$, and the remaining conditions, concerning the diagonal terms, are as follows:
$$\sum_{a_1\ldots a_s}\delta_\pi\begin{pmatrix}a_1&\ldots&a_s&a_1&\ldots&a_s\\ b_1&\ldots&b_s&b_1&\ldots&b_s\end{pmatrix}=c,\forall b$$

As a first observation, the quantity on the left is a decreasing function of $\lambda=\ker b$. Now in order for this decreasing function to be constant, we must have:
$$\sum_{a_1\ldots a_s}\delta_\pi\begin{pmatrix}a_1&\ldots&a_s&a_1&\ldots&a_s\\ 1&\ldots&s&1&\ldots&s\end{pmatrix}=\sum_{a_1\ldots a_s}\delta_\pi\begin{pmatrix}a_1&\ldots&a_s&a_1&\ldots&a_s\\ 1&\ldots&1&1&\ldots&1\end{pmatrix}$$

We conclude that the condition $[^\mu_\pi]\leq\mu$ must be strengthened into $[^\mu_\pi]=\mu$, as claimed. Finally, the last assertion is clear, by using either (1) or (2).
\end{proof}

In the symmetric case, $\pi=\pi^\circ$, we have the following result:

\begin{proposition}
Given a partition $\pi\in P(2s,2s)$ which is symmetric, $\varphi_\pi$ is unital modulo scalars precisely when its symmetric components are as follows,
\begin{enumerate}
\item Symmetric blocks with $v\leq 1$,

\item Unions of asymmetric blocks with $r+u=0,v+w=1$,

\item Unions of asymmetric blocks with $r+u\geq1,v+w\leq1$,
\end{enumerate}
with the conventions from Proposition 3.3 for the values of $r,u,v,w$.
\end{proposition}

\begin{proof}
We know from Proposition 5.2 that the condition in the statement is equivalent to $[^\mu_\pi]=\mu$, and we can see from this that $\pi$ satisfies the condition if and only if all the symmetric components of $\pi$ satisfy the condition. Thus, we must simply check the validity of $[^\mu_\pi]=\mu$ for the partitions in Proposition 3.3, and this gives the result.

To be more precise, for the 1-block components the study is trivial, and we are led to (1). Regarding now the 2-block components, in the case $r+u=0$ we must have $v+w=1$, as stated in (2). Finally, assuming $r+u\geq1$, when constructing $[^\mu_\pi]$ all the legs on the bottom will become connected, and so we must have $v+w\leq1$, as stated in (3).
\end{proof}

Summarizing, the condition that $\varphi_\pi$ is unital modulo scalars is a natural generalization of what happens for the 4 basic partitions in $P_{even}(2,2)$, and in the symmetric case, we have a good understanding of such partitions. However, the associated matrices $\Lambda_\pi$ still fail to be multiplicative, and we must come up with a second condition, coming from:

\begin{theorem}
Assuming that $\pi\in P(2s,2s)$ is symmetric, the following conditions are equivalent:
\begin{enumerate}
\item The linear maps $\varphi_\pi,\varphi_{\pi^*}$ are both unital modulo scalars.

\item The symmetric components have $\leq2$ upper legs, and $\leq2$ lower legs.

\item The symmetric components appear as copies of the $4$ elements of $P_{even}(2,2)$.
\end{enumerate}
\end{theorem}

\begin{proof}
By applying Proposition 5.3 above to the partitions $\pi,\pi^*$, and by merging the results, we conclude that the equivalence $(1)\iff(2)$ holds indeed. 

As for the equivalence $(2)\iff(3)$, this is clear from definitions.
\end{proof}

Summarizing, we have now the needed combinatorial ingredients for dealing with the non-multiplicativity issues mentioned after Proposition 5.1 above.

\section{Poisson laws}

In this section we put together everything that we have. The idea will be that of using the partitions found in Theorem 5.4 above as an input for Proposition 5.1, and then for the general block-modification machinery developed in section 1.

We will need the following result, which complements Proposition 2.4: 

\begin{proposition}
The following functions $\varphi:NC_p\times NC_p\to\mathbb R$ are ``multiplicative'', in the sense that they satisfy the condition $\varphi(\sigma,\gamma)=\varphi(\sigma,\sigma)$:
\begin{enumerate}
\item $\varphi(\sigma,\tau)=|\tau\sigma|-|\tau|$.

\item $\varphi(\sigma,\tau)=|\tau^{-1}\sigma|-|\tau|$.
\end{enumerate}
\end{proposition}

\begin{proof}
These results follow indeed from the following computations:
\begin{eqnarray*}
\varphi_1(\sigma,\gamma)&=&|\gamma\sigma|-1=|\sigma^2|-|\sigma|=\varphi_1(\sigma,\sigma)\\
\varphi_2(\sigma,\gamma)&=&|\gamma^{-1}\sigma|-1=p-|\sigma|=\varphi_2(\sigma,\sigma)
\end{eqnarray*}

To be more precise, here we have used in (1) a result from \cite{bn2}, stating that the numbers $|\gamma\sigma|-1$ and $|\sigma^2|-|\sigma|$ are equal, both counting the number of blocks of $\sigma$ having even size. As for (2), this comes from the formula $|\sigma\gamma^{-1}|-1=p-|\sigma|$ from \cite{bia}, and from the fact that $\sigma\gamma^{-1},\gamma^{-1}\sigma$ have the same cycle structure as the left and right Kreweras complements of $\sigma$, and therefore have the same number of blocks. See \cite{bn2}.
\end{proof}

We can now formulate our main multiplicativity result, as follows:

\begin{proposition}
Assuming that $\pi\in P_{even}(2s,2s)$ is symmetric, $\pi=\pi^\circ$, and is such that $\varphi_\pi,\varphi_{\pi^*}$ are unital modulo scalars, we have a formula of the following type:
$$(M_\sigma\otimes M_\tau)(\Lambda_\pi)=N^{a+b|\sigma|+c|\tau|+d|\sigma\wedge\tau|+e|\sigma\tau|+f|\sigma\tau^{-1}|+g|\tau\sigma|+h|\tau^{-1}\sigma|}$$
Moreover, the square matrix $\Lambda_\pi$ is multiplicative, in the sense of Definition 1.7.
\end{proposition}

\begin{proof}
The first assertion follows from Proposition 5.1. Indeed, according to Theorem 5.4 above, the list of partitions appearing in Proposition 5.1 (2) dissapears in the case where both $\varphi_\pi,\varphi_{\pi^*}$ are unital modulo scalars, and this gives the result.

As for the second assertion, this follows from the formula in the statement, and from the various results in Proposition 2.4 and Proposition 6.1 above.
\end{proof}

As a main consequence, Theorem 1.8 applies, and gives:

\begin{theorem}
Given a partition $\pi\in P_{even}(2s,2s)$ which is symmetric, $\pi=\pi^\circ$, and which is such that $\varphi_\pi,\varphi_{\pi^*}$ are unital modulo scalars, for the corresponding block-modified Wishart matrix $\tilde{W}=(id\otimes\varphi_\pi)W$ we have the asymptotic convergence formula
$$m\tilde{W}\sim\pi_{mn\rho}$$
in $*$-moments, in the $d\to\infty$ limit, where $\rho=law(\Lambda_\pi)$.
\end{theorem}

\begin{proof}
This follows by putting together the results that we have. Indeed, due to Proposition 6.2 above, Theorem 1.8 applies, and gives the convergence result. 
\end{proof}

Summarizing, we have now an explicit block-modification machinery, valid for certain suitable partitions $\pi\in P_{even}(2s,2s)$, which improves the previous theory from \cite{bn2}.

There are of course many interesting questions left, in relation with the easy symmetric modifications, that we will not further investigate here, as follows:
\begin{enumerate}
\item A first question is that of computing the above measure $\rho=law(\Lambda_\pi)$. Here we know that the moments of $\Lambda_\pi$ are the quantities $(M_\gamma\otimes M_\gamma)(\Lambda_\pi)$, and by using the computations in the proof of Theorem 2.6, this shows that $\rho$ is a certain sum of three Dirac masses on the real line, depending on three parameters. 

\item A more delicate question now, which is of interest in connection with quantum information theory questions, is that of computing the support of $\pi_{mn\rho}$, and deciding when this support is contained in $[0,\infty)$. For sums $\rho$ of two Dirac masses, this was done in \cite{bn1}. In the present case, we do not know the answer.

\item As a quite difficult combinatorial question now, we have the problem of understanding the combinatorics of the partitions that we excluded, from Proposition 5.1 (2). As already mentioned, these partitions do not seem to have any ``obvious'' multiplicativity type property, even in the generalized sense of \cite{bn2}.
\end{enumerate}

Finally, we have the question of investigating the partitions $\pi\in P_{even}(2s,2s)$ which are not symmetric, or more generally the partitions $\pi\in P(2s,2s)$, without assumptions. Regarding this latter question, here we have some computations at $s=1$, whose outcome is partly good, partly bad. We have no conjecture here, for the moment. 

\section{Twisted easiness}

In this section and in the next one we ``twist'' the result in Theorem 6.3, by using the general Schur-Weyl twisting theory from \cite{ba1}. Our motivation comes from our belief that the free Bessel laws introduced in \cite{bb+} can be covered by a suitable ``super-easy'' extension of our formalism. All this will be explained later on, in sections 9-10 below.

In order to start the twisting work, let us first recall from \cite{ba1} that we have:

\begin{proposition}
We have a signature map $\varepsilon:P_{even}\to\{-1,1\}$, given by $\varepsilon(\pi)=(-1)^c$, where $c$ is the number of switches needed to make $\pi$ noncrossing. In addition:
\begin{enumerate}
\item For $\pi\in Perm(k,k)\simeq S_k$, this is the usual signature.

\item For $\pi\in P_2$ we have $(-1)^c$, where $c$ is the number of crossings.

\item For $\pi\in P$ obtained from $\sigma\in NC_{even}$ by merging blocks, the signature is $1$.
\end{enumerate}
\end{proposition}

\begin{proof}
The fact that the number $c$ in the statement is well-defined modulo 2 is standard, and we refer here to \cite{ba1}. As for the remaining assertions, these are as well from \cite{ba1}.
\end{proof}

We can make act partitions in $P_{even}$ on tensors in a twisted way, as follows:

\begin{definition}
Associated to any partition $\pi\in P_{even}(k,l)$ is the linear map
$$\bar{T}_\pi(e_{i_1}\otimes\ldots\otimes e_{i_k})=\sum_{\sigma\leq\pi}\varepsilon(\sigma)\sum_{j:\ker(^i_j)=\sigma}e_{j_1}\otimes\ldots\otimes e_{j_l}$$
where $\varepsilon:P_{even}\to\{-1,1\}$ is the signature map.
\end{definition}

Observe the similarity with the formula in Definition 2.1, from the untwisted case. In fact, what we are doing here is basically replacing the Kronecker symbols $\delta_\pi\in\{0,1\}$ from the untwisted case with certain signed Kronecker symbols, $\bar{\delta}_\pi\in\{-1,0,1\}$.

Combinatorially speaking, this modification is something quite subtle, and as a main conceptual result here, we have the following statement, which is also from \cite{ba1}:

\begin{proposition}
The assignement $\pi\to\bar{T}_\pi$ is categorical, in the sense that
$$\bar{T}_\pi\otimes\bar{T}_\sigma=\bar{T}_{[\pi\sigma]}\quad,\quad\bar{T}_\pi\bar{T}_\sigma=N^{c(\pi,\sigma)}\bar{T}_{[^\sigma_\pi]}\quad,\quad\bar{T}_\pi^*=\bar{T}_{\pi^*}$$
where $c(\pi,\sigma)$ is the number of closed loops obtained when composing.
\end{proposition}

\begin{proof}
All this is routine, by using Proposition 7.1, and we refer here to \cite{ba1}.
\end{proof}

The above result is important for representation theory reasons, because it shows that Definition 7.2 can be used in order to construct new quantum groups. See \cite{ba1}.  

In relation now with our problems, given a partition $\pi\in P_{even}(2s,2s)$, we can use if we want the twisted version $\bar{\varphi}_\pi$ of the linear map $\varphi_\pi$ constructed in section 2. To be more precise, we have the following twisted analogue of Definition 2.3 above:

\begin{definition}
Associated to any partition $\pi\in P_{even}(2s,2s)$ is the linear map
$$\bar{\varphi}_\pi(e_{a_1\ldots a_s,c_1\ldots c_s})=\sum_{\sigma\leq\pi}\varepsilon(\sigma)\sum_{\ker(^{ac}_{bd})=\sigma}e_{b_1\ldots b_s,d_1\ldots d_s}$$
obtained from $\bar{T}_\pi$ by contracting all the tensors, via $e_{i_1}\otimes\ldots\otimes e_{i_{2s}}\to e_{i_1\ldots i_s,i_{s+1}\ldots i_{2s}}$.
\end{definition}

Observe that, in contrast to Definition 2.3, the assumption that $\pi\in P(2s,2s)$ must actually belong to $P_{even}(2s,2s)$ is really needed here, because the signature map $\varepsilon$ is only defined for such partitions, and cannot be extended to the whole $P(2s,2s)$.

As an important remark, due to Proposition 7.1 (3) above, for $\pi\in NC_{even}(2s,2s)$ we have $\bar{\varphi}_\pi=\varphi_\pi$. In other words, in order to reach to some new problems and some new combinatorics, we must restrict the attention to the partitions having crossings. 

Let us begin with a complete study at $s=1$. As a first result here, we have:

\begin{proposition}
For the only partition $\pi\in P_{even}(2,2)$ having crossings, namely
$$\pi=\begin{bmatrix}\circ&\bullet\\ \bullet&\circ\end{bmatrix}$$
we have $\bar{\varphi}_\pi(A)=2A^\delta-A^t$, and the generalized moments of $\bar{\Lambda}_\pi$ are given by
$$(M_\sigma\otimes M_\tau)(\bar{\Lambda}_\pi)=\frac{1}{n^{|\sigma|+|\tau|}}\sum_{\ker i\leq\sigma\tau}(-1)^{\varepsilon_{i_1i_{\sigma(1)}}+\ldots+\varepsilon_{i_pi_{\sigma(p)}}}$$
where $\varepsilon_{ac}=1-\delta_{ac}$, which equals $1$ when $a\neq c$, and equals $0$ when $a=c$.
\end{proposition}

\begin{proof}
The formula in Definition 7.4 above gives $\bar{\varphi}_\pi(A)=2A^\delta-A^t$, as claimed. As for the corresponding square matrix $\bar{\Lambda}_\pi$, this is given by the following formula:
$$\bar{\Lambda}_\pi(e_a\otimes e_c)=
\begin{cases}
e_a\otimes e_a&{\rm if}\ a=c\\
-e_c\otimes e_a&{\rm if}\ a\neq c
\end{cases}$$

In terms of the opposite Kronecker symbols $\varepsilon_{ac}=1-\delta_{ac}$, as in the statement, this formula can be written in a global way, as follows:
$$\bar{\Lambda}_\pi(e_a\otimes e_c)=(-1)^{\varepsilon_{ac}}e_c\otimes e_a$$

Now in terms of matrix coefficients, we therefore have:
$$(\bar{\Lambda}_\pi)_{ab,cd}=(-1)^{\varepsilon_{ac}}\delta_{ad}\delta_{bc}$$

According now to Definition 1.2, the quantities that we are interested in are:
\begin{eqnarray*}
&&(M_\sigma\otimes M_\tau)(\bar{\Lambda}_\pi)\\
&=&\frac{1}{n^{|\sigma|+|\tau|}}\sum_{i_1\ldots i_p}\sum_{j_1\ldots j_p}(-1)^{\varepsilon_{i_1i_{\sigma(1)}}+\ldots+\varepsilon_{i_pi_{\sigma(p)}}}
\delta_{i_1j_{\tau(1)}}\ldots\delta_{i_pj_{\tau(p)}}
\delta_{j_1i_{\sigma(1)}}\ldots\delta_{j_pi_{\sigma(p)}}\\
&=&\frac{1}{n^{|\sigma|+|\tau|}}\sum_{\ker i\leq\sigma\tau}(-1)^{\varepsilon_{i_1i_{\sigma(1)}}+\ldots+\varepsilon_{i_pi_{\sigma(p)}}}
\end{eqnarray*}

Thus, we have obtained the formula in the statement, and we are done.
\end{proof}

Let us compute now the law of $\bar{\Lambda}_\pi$. According to the general twisting philosophy in \cite{ba1}, the law of $\Lambda_\pi$ should be invariant under twisting, and this is indeed the case:

\begin{proposition}
For the standard crossing $\pi\in P_{even}(2,2)$, we have
$$law(\bar{\Lambda}_\pi)=law(\Lambda_\pi)$$
and this law is the measure $\rho=\frac{n-1}{2n}\delta_{-1}+\frac{n+1}{2n}\delta_1$ from the proof of Theorem 2.6 (2).
\end{proposition}

\begin{proof}
According to the general theory from section 1 above, the moments of $\bar{\Lambda}_\pi$ can be obtained by setting $\sigma=\tau=\gamma$ in the formula from Proposition 7.5. 

Thus, we must compute the following quantities:
$$(M_\gamma\otimes M_\gamma)(\bar{\Lambda}_\pi)=\frac{1}{n^2}\sum_{\ker i\leq\gamma^2}(-1)^{\varepsilon_{i_1i_2}+\varepsilon_{i_1i_3}+\ldots+\varepsilon_{i_pi_1}}$$

Our claim is that the $\pm1$ numbers appearing on the right are in fact all equal to $1$. Indeed, we have two cases here, depending on whether $p$ is odd or even:

(1) When $p$ is odd the square of $\gamma=(1\to2\to\ldots\to p\to1)$ is once again a full cycle, $\gamma^2=(1\to3\to\ldots\to p\to2\to4\to\ldots\to p-1\to1)$, and so the condition $\ker i\leq\gamma^2$ simply tells us that all the indices of $i$ are equal. Thus, the above opposite Kronecker symbols are all equal to $0$, and so is their sum, and we obtain $1$ as summand.

(2) When $p$ is even the square of $\gamma=(1\to2\to\ldots\to p\to1)$ decomposes as a union of two cycles, namely $(1\to3\to\ldots\to p-1\to1)$ and $(2\to4\to\ldots\to p\to2)$, and so our condition $\ker i\leq\gamma^2$  reads $i_1=i_3=\ldots=i_{p-1}=\alpha$ and $i_2=i_4=\ldots=i_p=\beta$, for certain numbers $\alpha,\beta$. Thus, the above opposite Kronecker symbols have the same value, namely $\varepsilon_{\alpha\beta}$, and since we are summing an even number of them, their global sum equals 0 modulo 2, and we obtain once again 1 as summand, in our formula.

Summarizing, we have proved our claim. Thus, the signatures dissapear in the computation, and by comparing with the computation for $\Lambda_\pi$, from the proof of Proposition 2.5 above, we conclude that we obtain the same law as there, as claimed.
\end{proof}

We can now formulate our final result regarding the crossing, as follows:

\begin{theorem}
For the only partition $\pi\in P_{even}(2,2)$ having crossings, namely
$$\pi=\begin{bmatrix}\circ&\bullet\\ \bullet&\circ\end{bmatrix}$$
we have $\bar{\varphi}_\pi(A)=2A^\delta-A^t$, and the asymptotic distribution of the corresponding block-modified Wishart matrix is the same as the one coming from the map $\varphi_\pi(A)=A^t$.
\end{theorem}

\begin{proof}
The first part of the statement, which is there just for completness, is from Proposition 7.5. In order to prove now that Theorem 1.8 applies, we must check the multiplicativity condition $(M_\sigma\otimes M_\gamma)(\bar{\Lambda}_\pi)=(M_\sigma\otimes M_\sigma)(\bar{\Lambda}_\pi)$ from Definition 1.7. By using the general formula from Proposition 7.5, the equality that we want to prove reads:
$$\frac{1}{n^{|\sigma|+1}}\sum_{\ker i\leq\sigma\gamma}(-1)^{\varepsilon_{i_1i_{\sigma(1)}}+\ldots+\varepsilon_{i_pi_{\sigma(p)}}}
=\frac{1}{n^{2|\sigma|}}\sum_{\ker i\leq\sigma^2}(-1)^{\varepsilon_{i_1i_{\sigma(1)}}+\ldots+\varepsilon_{i_pi_{\sigma(p)}}}$$

Let us denote the exponents appearing in this formula by $e_\sigma(i)$. With this notation, and by rescaling, the formula that we want to prove is as follows:
$$\frac{1}{n}\sum_{\ker i\leq\sigma\gamma}(-1)^{e_\sigma(i)}
=\frac{1}{n^{|\sigma|}}\sum_{\ker i\leq\sigma^2}(-1)^{e_\sigma(i)}$$

Now let us recall that in the untwisted case, where all the summands are 1, which formally means setting $e_\sigma=0$ in the above formula, the above equality holds indeed, because this follows from the formula $|\sigma\gamma|-1=|\sigma^2|-|\sigma|$ established in \cite{bn2}, by using the fact that both these numbers count the number of even blocks of $\sigma$.

In the twisted case now, the proof is similar. Indeed, the bijection established in \cite{bn2} with the even blocks of $\sigma$ preserves, at the level of the corresponding indices, the above $e_\sigma(i)$ numbers. Thus when summing, we obtain again the same quantities.

Summarizing, the matrix $\bar{\Lambda}_\pi$ is multiplicative, and so Theorem 1.8 applies. But, in view of Proposition 7.6, we obtain the same law as in the untwisted case, as claimed.
\end{proof}

As already mentioned before Proposition 7.6, such kind of results are in tune with the general Schur-Weyl twisting philosophy in \cite{ba1}, which states that the main combinatorial invariants, such as the Weingarten functions of the corresponding orthogonal groups, should be invariant under twisting. We will see in the next section that this phenomenon extends well beyond the standard crossing setting, into a fully general result. 

\section{The general case}

In this section, we investigate the general twisted case. For this purpose, we will basically follow the decomposition method from the proof of Theorem 6.3 above, with twisting input coming from the various formulae obtained in section 7.

Let us begin with some preliminary results. We first have:

\begin{proposition}
The signature of a symmetric partition $\pi\in P_{even}(2s,2s)$ is given by
$$\varepsilon(\pi)=\prod_t\varepsilon(\pi_t)$$
where $\pi_t$ are the symmetric components of $\pi$.
\end{proposition}

\begin{proof}
Since our partition satisfies $\pi\in P_{even}(2s,2s)$ and $\pi=\pi^\circ$, its components satisfy as well these conditions, $\pi_t\in P_{even}(2s,2s)$ and $\pi_t=\pi_t^\circ$. In particular, any of these partitions $\pi_t$ has the property that both the number of upper legs, and of lower legs, is even.

Now let us rearrange these blocks, as for our partition $\pi$ to become the usual tensor product $\pi'=\pi_1\otimes\ldots\otimes\pi_t$. Due to the above remark regarding the upper and lower legs of each $\pi_t$, the number of required switches for this operation is even, and so the signature does not change. In other words, we have $\varepsilon(\pi)=\varepsilon(\pi')$, and this gives the result.
\end{proof}

We will need as well the following technical observation:

\begin{proposition}
For the standard pairing $\eta\in P_{even}(2s,2s)$ having horizontal strings,
$$\eta=\begin{bmatrix}
a&b&c&\ldots&a&b&c&\ldots\\
\alpha&\beta&\gamma&\ldots&\alpha&\beta&\gamma&\ldots
\end{bmatrix}$$
we have $\bar{\Lambda}_\eta=\Lambda_\eta$, and so $(M_\sigma\otimes M_\tau)(\bar{\Lambda}_\eta)=(M_\sigma\otimes M_\tau)(\Lambda_\eta)=1$, for any $\sigma,\tau$.
\end{proposition}

\begin{proof}
The equality $\bar{\Lambda}_\eta=\Lambda_\eta$ is clear at $s=1$, and at $s=2$ this follows from the fact that $\eta$ needs 2 switches in order to be made noncrossing, and so its signature is 1. In general, this follows from the fact that $\eta$, as well as all its subpartitions, all have signature 1. As for the last assertion, this is clear from $(M_\sigma\otimes M_\tau)(\Lambda_\eta)=1$, for any $\sigma,\tau$.
\end{proof}

With the above results in hand, we have the following analogue of Theorem 3.5:

\begin{proposition}
Given a symmetric partition $\pi\in P_{even}(2s,2s)$, the quantities
$$(M_\sigma\otimes M_\tau)(\bar{\Lambda}_\pi)$$
decompose as elementary products, over the symmetric blocks of $\pi$.
\end{proposition}

\begin{proof}
We follow the proof from the untwisted case, from section 3 above. First, in the twisted case, the formula from Proposition 3.1 becomes:
\begin{eqnarray*}
(\bar{\Lambda}_\pi)_{a_1\ldots a_s,b_1\ldots b_s,c_1\ldots c_s,d_1\ldots d_s}
&=&\varepsilon\begin{bmatrix}a_1&\ldots&a_s&c_1&\ldots&c_s\\ b_1&\ldots&b_s&d_1&\ldots&d_s\end{bmatrix}\\
&&\times\,\delta_\pi\begin{pmatrix}a_1&\ldots&a_s&c_1&\ldots&c_s\\ b_1&\ldots&b_s&d_1&\ldots&d_s\end{pmatrix}
\end{eqnarray*}

Regarding now the formula in Proposition 3.2, its twisted analogue, in the case of the self-adjoint partitions, that we are interested in here, is as follows:
\begin{eqnarray*}
(M_\sigma\otimes M_\tau)(\bar{\Lambda}_\pi)
&=&\frac{1}{n^{|\sigma|+|\tau|}}\sum_{i_1^1\ldots i_p^s}\sum_{j_1^1\ldots j_p^s}\prod_x\varepsilon\begin{bmatrix}i_x^1&\ldots&i_x^s&i_{\sigma(x)}^1&\ldots&i_{\sigma(x)}^s\\
j_x^1&\ldots&j_x^s&j_{\tau(x)}^1&\ldots&j_{\tau(x)}^s\end{bmatrix}\\
&&\hskip35mm\times\,\delta_\pi\begin{pmatrix}i_x^1&\ldots&i_x^s&i_{\sigma(x)}^1&\ldots&i_{\sigma(x)}^s\\
j_x^1&\ldots&j_x^s&j_{\tau(x)}^1&\ldots&j_{\tau(x)}^s\end{pmatrix}
\end{eqnarray*}

Consider the arrays of multi-indices appearing on the left, namely:
$$A_x=\begin{pmatrix}i_x^1&\ldots&i_x^s&i_{\sigma(x)}^1&\ldots&i_{\sigma(x)}^s\\
j_x^1&\ldots&j_x^s&j_{\tau(x)}^1&\ldots&j_{\tau(x)}^s\end{pmatrix}$$

In terms of these arrays, our moment formula simply becomes:
$$(M_\sigma\otimes M_\tau)(\bar{\Lambda}_\pi)
=\frac{1}{n^{|\sigma|+|\tau|}}\sum_{[A_x]\leq\pi,\forall x}\varepsilon[A_x]$$

Now observe that, if we denote by $C_t\subset\{1,\ldots,s\}$ the set of left columns affected by $\pi_t$, so that $C_t+s$ is the set of right columns affected by $\pi_t$, the conditions $[A_x]\leq\pi,\forall x$ split over the symmetric components $\pi_t$, because these conditions affect precisely the $C_t$-th columns of the arrays of indices $i,j$. Thus, our summation decomposes as follows, where $K$ is a certain normalization constant, coming from the horizontal strings:
$$(M_\sigma\otimes M_\tau)(\bar{\Lambda}_\pi)
=\frac{K}{n^{|\sigma|+|\tau|}}\sum_{[A_x^t]\leq\pi_t,\forall x,t}\varepsilon[A_x]$$

We must prove now that the signatures on the right decompose as:
$$\varepsilon[A_x]=\prod_t\varepsilon[A_x^t]$$

But this follows by performing precisely the switches that we used in the proof of Proposition 8.2 above, whose total number is even, and which therefore do not change the signature. We conclude that we have the following formula:
\begin{eqnarray*}
(M_\sigma\otimes M_\tau)(\bar{\Lambda}_\pi)
&=&\frac{K}{n^{|\sigma|+|\tau|}}\sum_{[A_x^t]\leq\pi_t,\forall x,t}\prod_t\varepsilon[A_x^t]\\
&=&\frac{K}{n^{|\sigma|+|\tau|}}\prod_t\sum_{[A_x^t]\leq\pi_t,\forall x}\varepsilon[A_x^t]
\end{eqnarray*}

Thus, we have a decomposition over the symmetric components of $\pi$, as indicated in the statement. As for the computation of the normalization factor, we can use here Proposition 8.2, and we conclude that this factor is 1, as in the untwisted case.
\end{proof}

With these ingredients in hand, we can now prove:

\begin{theorem}
Assuming that $\pi\in P_{even}(2s,2s)$ is symmetric, and is such that $\varphi_\pi,\varphi_{\pi^*}$ are both unital modulo scalars, the $d\to\infty$ asymptotic law of the associated block-modified Wishart matrix $\tilde{W}=(id\otimes\bar{\varphi}_\pi)W$ is the same as in the untwisted case.
\end{theorem}

\begin{proof}
Let us first check that $\bar{\Lambda}_\pi$ is multiplicative. By using Proposition 8.3 we have to check this over the various symmetric components of $\pi$. On the other hand, we know from Theorem 5.3 above that these symmetric components appear as copies of the 4 basic partitions, from $P_{even}(2,2)$. Thus, the result follows from Theorem 7.7.

Summarizing, Theorem 1.8 applies. Regarding now the $d\to\infty$ asymptotic law that we obtain, we must work out here the twisted analogue of the formula in Proposition 6.2. But, once again due to Theorem 7.7, this formula must be invariant under twisting, and we are done. Note that this is in tune with the general twisting philosophy from \cite{ba1}.
\end{proof}

Summarizing, we have now a block-modification theory for some suitable partitions $\pi\in P_{even}(2s,2s)$, which works both in the untwisted and the twisted case.

As already mentioned, our main motivation for this extension comes from a potential application to the free Bessel laws, and this will be discussed in what follows. However, Theorem 8.4 has as well its own interest, supporting the twisting philosophy from \cite{ba1}.

\section{Free Bessel laws}

In this section and in the next one we go back to the general complex setting, with the aim of recovering the free Bessel laws, constructed in \cite{bb+}, by using our formalism.

The free Bessel laws are compound free Poisson laws, constructed as follows:

\begin{definition}
The free Bessel law, depending on $n\in\mathbb N\cup\{\infty\}$ and on $t>0$, is
$$\beta^n_t=\pi_{t\eta_n}$$
where $\eta_n$ is the uniform measure on the $n$-th roots of unity.
\end{definition}

The terminology here comes from a similarity with the Bessel laws, also introduced in \cite{bb+}, which is best understood in terms of the Bercovici-Pata bijection \cite{bpa}. At $n=1$ we obtain the free Poisson laws, while at $n=2$ we obtain the laws constructed and studied \cite{bbc}. In general, the free Bessel laws appear as laws of truncated characters of certain compact quantum Lie groups in the sense of Woronowicz \cite{wor}.

As explained in \cite{bb+}, the moments of $\beta^n_1$ can be recovered by using the certain block-modified Wishart matrices, with the modification map being the one which multiplies by a diagonal matrix whose eigenvalues are uniformly distributed $n$-th roots of unity.

By using the methods in section 1 above, we can now perform a finer study of this phenomenon, at the $*$-moment level, and at any $t>0$. First, we have: 

\begin{proposition}
For the block-modified matrix $\tilde{W}=(id\otimes\varphi)W$, where $\varphi(A)=EA$, with $E=diag(1,w,\ldots,w^{n-1})$ with $w=e^{2\pi i/n}$, we have
$$\lim_{d\to\infty}M_p^e\left(\tilde{W}\right)=\sum_{\tau\in NC_p^e(n)}\left(\frac{n}{m}\right)^{|\tau|-1}$$
where $NC_p^e(n)\subset NC_p$ is the set of partitions $\pi$ having the property that any block of $\pi$, when weighted by $e$, must have as size a multiple of $n$. 
\end{proposition}

\begin{proof}
We use Proposition 1.4 above. For the linear map $\varphi(A)=EA$ in the statement we have $\Lambda_{ab,cd}=\delta_{ab}\delta_{cd}w^a$. Now with the usual convention that the exponents $e_i\in\{1,*\}$ get converted into signs $e_i\in\{1,-1\}$, we obtain the following formula:
\begin{eqnarray*}
(M_\sigma^e\otimes M_\tau^e)(\Lambda)
&=&\frac{1}{n^{|\sigma|+|\tau|}}\sum_{i_1\ldots i_p}\delta_{i_{\sigma(1)}i_{\tau(1)}}\ldots\delta_{i_{\sigma(p)}i_{\tau(p)}}w^{e_1i_1+\ldots+e_pi_p}\\
&=&\frac{1}{n^{|\sigma|+|\tau|}}\sum_{\ker i\leq\sigma\tau^{-1}}w^{e_1i_1+\ldots+e_pi_p}
\end{eqnarray*}

By decomposing now over the blocks of $\sigma\tau^{-1}$, we obtain from this:
\begin{eqnarray*}
(M_\sigma^e\otimes M_\tau^e)(\Lambda)
&=&\frac{1}{n^{|\sigma|+|\tau|}}\prod_{\beta\in\sigma\tau^{-1}}\sum_iw^{i(\sum_{b\in\beta}e_b)}\\
&=&\frac{1}{n^{|\sigma|+|\tau|}}\prod_{\beta\in\sigma\tau^{-1}}n\cdot \delta_{\sum_{b\in\beta}e_b=0(n)}\\
&=&n^{|\sigma\tau^{-1}|-|\sigma|-|\tau|}\prod_{\beta\in\sigma\tau^{-1}}\delta_{\sum_{b\in\beta}e_b=0(n)}
\end{eqnarray*}

As a conclusion, modulo a rescaling factor, the generalized $*$-moments that we are interested in are in fact Kronecker type symbols, which decide whether all blocks of $\sigma\tau^{-1}$ have as size a multiple of $n$, when weighted by the signs $e_1,\ldots,e_p\in\{-1,1\}$.

Now if we denote these latter Kronecker symbols by $\delta_e(\sigma\tau^{-1})$, which depend on the value of $n\in\mathbb N$ as well, our formula for the generalized $*$-moments simply becomes:
$$(M_\sigma^e\otimes M_\tau^e)(\Lambda)=n^{|\sigma\tau^{-1}|-|\sigma|-|\tau|}\delta_e(\sigma\tau^{-1})$$

We can now apply Proposition 1.4. By using this result, along with the standard formula $|\sigma|+|\sigma\gamma^{-1}|=p+1$ from \cite{bia}, we obtain the following formula:
\begin{eqnarray*}
\lim_{d\to\infty}M_p^e\left(m\tilde{W}\right)
&=&\sum_{\sigma\in NC_p}(mn)^{|\sigma|}n^{|\sigma\gamma^{-1}|-|\sigma|-1}\delta_e(\sigma\gamma^{-1})\\
&=&\sum_{\sigma\in NC_p}m^{|\sigma|}n^{|\sigma\gamma^{-1}|-1}\delta_e(\sigma\gamma^{-1})\\
&=&\sum_{\sigma\in NC_p}m^{|\sigma|}n^{p-|\sigma|}\delta_e(\sigma\gamma^{-1})
\end{eqnarray*}

We recall now from \cite{bn2} that for $\sigma\in NC_p$, the partition $\sigma\gamma^{-1}\in P_p$ has the same block structure as the Kreweras complement $\widehat{\sigma}\in NC_p$. Thus we can replace at right $\sigma\gamma^{-1}\to\widehat{\sigma}$, and by resumming over the partition $\tau=\widehat{\sigma}$, we can write the above formula as:
\begin{eqnarray*}
\lim_{d\to\infty}M_p^e\left(m\tilde{W}\right)
&=&\sum_{\sigma\in NC_p}m^{|\sigma|}n^{p-|\sigma|}\delta_e(\widehat{\sigma})\\
&=&\sum_{\tau\in NC_p}m^{p+1-|\tau|}n^{|\tau|-1}\delta_e(\tau)\\
&=&m^p\sum_{\tau\in NC_p}\left(\frac{n}{m}\right)^{|\tau|-1}\delta_e(\tau)
\end{eqnarray*}

By dividing now by $m^p$, we obtain the following formula:
$$\lim_{d\to\infty}M_p^e\left(\tilde{W}\right)=\sum_{\tau\in NC_p}\left(\frac{n}{m}\right)^{|\tau|-1}\delta_e(\tau)$$

But this gives the formula in the statement, and we are done.
\end{proof}

We can now review the random matrix result from \cite{bb+}. With a few enhancements, which consist on one hand in dropping the assumption $m=n$ used there, and on the other hand in commenting as well on the $*$-moments, the result is as follows:

\begin{theorem}
Consider the block-modified matrix $\tilde{W}=(id\otimes\varphi)W$, where $\varphi(A)=EA$, with $E=diag(1,w,\ldots,w^{n-1})$, where $w=e^{2\pi i/n}$.
\begin{enumerate}
\item In the case $m=n$, in the $d\to\infty$ limit we have $\tilde{W}\simeq\beta^n_1$, in moments.

\item In general we obtain, modulo a Dirac mass at $0$, the measure $\beta^n_{n/m}$.

\item The above $d\to\infty$ convergences fail to hold in $*$-moments.
\end{enumerate}
\end{theorem}

\begin{proof}
The first assertion is from \cite{bb+}, obtained there by using an old method from \cite{glm}, and this can be proved as well by using Proposition 9.2 above. To be more precise, in the case $m=n$, and with $e=1$, the formula in Proposition 9.2 simply becomes:
$$\lim_{d\to\infty}M_p^e\left(\tilde{W}\right)=\# NC_p(n)$$

We conclude from this that the asymptotic free cumulants of $\tilde{W}$ are given by:
$$\kappa_p=\begin{cases}
1&{\rm if}\ n|p\\
0&{\rm otherwise}
\end{cases}$$

On the other hand, the free cumulants of the free Bessel law $\beta^n=\pi_{\eta_n}$ are the moments of the uniform measure $\eta_n$, and we therefore obtain the same numbers:
$$\kappa_p(\beta^n)=M_p(\eta_n)=\frac{1}{n}\sum_{k=1}^nw^{kp}=\begin{cases}
1&{\rm if}\ n|p\\
0&{\rm otherwise}
\end{cases}$$

In general now, we can write our moment formula as follows:
$$\lim_{d\to\infty}M_p^e\left(\tilde{W}\right)=\frac{m}{n}\sum_{\tau\in NC_p^e(n)}\left(\frac{n}{m}\right)^{|\tau|}$$

On the other hand, the cumulants of the free Bessel law $\beta^n_t=\pi_{t\eta_n}$ are the moments of the measure $t\eta_n$, and are therefore given by:
$$\kappa_p(\beta^n_t)=M_p(t\eta_n)=\frac{1}{n}\sum_{k=1}^nw^{kp}=\begin{cases}
t&{\rm if}\ n|p\\
0&{\rm otherwise}
\end{cases}$$

Thus, modulo a Dirac mass at 0, we obtain indeed the free Bessel law $\beta^n_{n/m}$.

Finally, regarding the last assertion, which is folklore, this follows by carefully examining the $*$-moments of order 4. Indeed, we have 16 moments to be examined, and the point is that those corresponding to the exponent $e=(1*1*)$ do not match.
\end{proof}

Summarizing, we have now an update and clarification of the random matrix material in \cite{bb+}, and the problem of finding a $*$-model for the free Bessel laws makes sense.

\section{Open problems}

In order to discuss the modelization problem for the free Bessel laws, open since \cite{bb+} in the $*$-moment setting, let us go back to the twisting considerations from sections 7-8. The twisted maps $\bar{\varphi}_\pi$ used there were obtained from the untwisted ones $\varphi_\pi$ by replacing the Kronecker symbols $\delta_\pi\in\{0,1\}$ by signed Kronecker symbols $\bar{\delta}_\pi\in\{-1,0,1\}$, the precise formula being as follows, with $\varepsilon:P_{even}\to\{-1,1\}$ being the signature map:
$$\bar{\varphi}_\pi(e_{a_1\ldots a_s,c_1\ldots c_s})=\sum_{\sigma\leq\pi}\varepsilon(\sigma)\sum_{\ker(^{ac}_{bd})=\sigma}e_{b_1\ldots b_s,d_1\ldots d_s}$$

One interesting question is that of further extending our formalism, by allowing the Kronecker symbols to take complex values, $\delta_\pi\in\mathbb T\cup\{0\}$. We believe that such an extended formalism can cover the free Bessel laws, in $*$-moments. However, the combinatorial theory here is not available yet, and this even in the real case, $\delta_\pi\in\{-1,0,1\}$. See \cite{ba2}.

Summarizing, in order to further advance, we would need here: (1) some more conceptual results on the free Bessel laws and their $*$-moments, coming for instance via \cite{anv}, and (2) some combinatorial advances on the notion of super-easiness, in the spirit of \cite{ba2}.

In addition to this question, and to the other questions raised throughout the paper, we have of course, as a central question, the problem of understanding the relationship between the block-modified Wishart matrices and the compact quantum groups. The situation here looks quite complicated, and we have for instance the following related question, that we believe to be of interest: ``what are the quantum groups having as spectral measure a compound free Poisson law?''. This is of course related to the free Bessel problem, via the quantum groups introduced and studied in \cite{bb+}.

We intend to come back to these questions in some future work.

\end{document}